\documentclass[reqno]{amsart}

\usepackage[ansinew]{inputenc}
\usepackage{amssymb}
\usepackage{fullpage}
\usepackage{amsmath}
\usepackage{graphicx}

\makeatletter \@addtoreset{equation}{section} \makeatother

\renewcommand\thefigure{\thesection.\@arabic\c@figure}
\renewcommand\thetable{\thesection.\@arabic\c@table}

\newtheorem{theorem}{Theorem}[section]
\newtheorem{lemma}[theorem]{Lemma}
\newtheorem{proposition}[theorem]{Proposition}
\newtheorem{corollary}[theorem]{Corollary}

\newtheorem{remark}[theorem]{Remark}

\newcommand{\bb}[1]{{\mathbb #1}}
\newcommand{\mc}[1]{{\mathcal #1}}
\newcommand{\mf}[1]{{\mathfrak #1}}
\newcommand{\mb}[1]{{\mathbf #1}}
\newcommand{\bs}[1]{{\boldsymbol #1}}

\newcommand{\<}{\langle}
\renewcommand{\>}{\rangle}

\newcommand{\cc}{\mathcal{C}}

\newcommand{\cm}{\mathcal{M}}

\newcommand{\R}{\mathbb{R}}

\let\G=\Gamma

\newcommand{\1}{\,\rlap{\small 1}\kern.13em 1}

\newcommand{\sqr}[2]{{\vcenter{\hrule height.#2pt%
                      \hbox{\vrule width.#2pt height#1pt\kern#1pt%
                            \vrule width.#2pt}%
                      \hrule height.#2pt}}}

\renewcommand{\limsup}{\mathop{\overline{\hbox{\rm lim}}}}
\renewcommand{\liminf}{\mathop{\underline{\hbox{\rm lim}}}}

\newcommand{\Tr}{\text{Tr}}

\title[Dynamical Large Deviations]{Dynamical
  large deviations for a boundary driven stochastic lattice gas model with many conserved quantities}

\author{Jonathan Farfan, Alexandre B. Simas \and Fábio J. Valentim}

\address{{\rm J. Farfan, A. B. Simas\and F. J. Valentim} \newline
IMPA, Estrada Dona Castorina 110,
CEP 22460 Rio de Janeiro, Brasil
\newline e-mail: \rm \texttt{jonathan@impa.br}, \texttt{alesimas@impa.br}, \texttt{valentim@impa.br}}
\begin{document}

\noindent  
\thanks{Research supported by CNPq.}

\begin{abstract}
We prove the dynamical large deviations for a particle system in which particles may have different velocities.  We assume that we have two infinite reservoirs of particles at the boundary: this is the so-called boundary driven process. The dynamics we considered consists of a weakly asymmetric simple exclusion process with collision among particles having different velocities. 
\end{abstract}

\maketitle


\section{introduction}
In the last years there has been considerable progress in understanding stationary non equilibrium states: reversible systems in contact
with different reservoirs at the boundary imposing a gradient on the conserved quantities of the system. In these systems there is a flow of matter
through the system and the dynamics is not reversible. The main difference with respect to equilibrium (reversible)
states is the following. In equilibrium, the invariant measure, which determines the
thermodynamic properties, is given for free by the Gibbs distribution specified by
the Hamiltonian. On the contrary, in non equilibrium states the construction of
the stationary state requires the solution of a dynamical problem. One of the most striking typical property of these systems is the presence of long-range correlations. For the symmetric simple exclusion this was already shown in a pioneering paper by Spohn \cite{S2}. 
We refer to \cite{BDGJL1,D} for two recent reviews on this topic. 

We discuss this issue in the context of stochastic lattice gases in a box of linear
size $N$ with birth and death process at the boundary modeling the reservoirs. We
consider the case when there are many thermodynamic variables: the local density
denoted by $\rho$, and the local momentum denoted by $p_k$, $k=1,\ldots,d$, $d$ being the dimension of the box.

The model which we will study can be informally described as follows: fix a velocity $v$, an integer $N\geq 1$, and boundary densities $0 < \alpha_v(\cdot) < 1$ and $0 < \beta_v(\cdot) < 1$; at any given time, each site of the set $\{1,\ldots,N-1\}\times\{0,\ldots,N-1\}^{d-1}$ is either empty or occupied by one particle at velocity $v$. In the bulk, each particle attempts to jump at any of its neighbors at the same velocity, with a weakly asymmetric rate. To respect the exclusion rule, the particle jumps only if the target
site at the same velocity $v$ is empty; otherwise nothing happens. At the boundary, sites with first coordinates given by $1$ or $N -1$ have particles
being created or removed in such a way that the local densities are $\alpha_v(\tilde{x})$ and $\beta_v(\tilde{x})$: at rate $\alpha_v(\tilde{x}/N)$ a particle is
created at $\{1\}\times\{\tilde{x}\}$ if the site is empty, and at rate $1-\alpha_v(\tilde{x})$ the particle at $\{1\}\times\{\tilde{x}\}$
is removed if the site is occupied, and at rate $\beta_v(\tilde{x})$ a particle is
created at $\{N-1\}\times\{\tilde{x}\}$ if the site is empty, and at rate $1-\beta_v(\tilde{x})$ the particle at $\{N-1\}\times\{\tilde{x}\}$
is removed if the site is occupied. Superposed to this dynamics, there is a collision process which exchange
velocities of particles in the same site in a way that momentum is conserved.

Similar models have been studied by \cite{BL,emy,qy}. In fact, the model we consider here is based on the model of Esposito et al. \cite{emy} which was used to derive the Navier-Stokes equation. It is also noteworthy that the derivation of hydrodynamic limits and macroscopic fluctuation theory for a system with two conserved quantities have been studied in \cite{cb}.

The hydrodynamic limit for the above model has been proved in \cite{s}. The hydrodynamic equation derives from the underlying stochastic dynamics through an appropriate scaling limit in which the microscopic time and space coordinates are rescaled diffusively. The hydrodynamic equation thus represents the law of
large numbers for the empirical density of the stochastic lattice gas. The convergence
has to be understood in probability with respect to the law of the stochastic
lattice gas.  Once it is established a natural question
is to consider large deviations.

This article thus provides a derivation of the dynamical large deviations for this model, and the proof follows the method introduced in \cite{flm}.
The main difference is that their proof of $I_T(\cdot|\gamma)$-density relied on some energy estimates that we were not able to achieve due to the presence of velocities. Therefore, we had to overcome problem by taking a different approach at that part.

The article is organized as follows: in Section \ref{sec2} we establish the notation and
state the main results of the article; in Section \ref{sec3}, we review the hydrodynamics for this model, that was obtained in \cite{s}; in Section \ref{sec4}, several properties of the rate function are derived; Section \ref{sec5} proves the $I_T(\cdot|\gamma)$-density, which is a key result for proving the lower bound; finally, in Section \ref{sec6} the proofs of the upper and lower bounds of the dynamical large deviations are given.

\section{Notation and Results}

\label{sec2}

Fix a positive integer $d\ge 1$.  Denote by ${D^d}$ the open set $(0,1)
\times \bb T^{d-1}$, where $\bb T^{k}$ is the $k$-dimensional torus
$[0,1)^k$, and by $\Gamma$ the boundary of ${D^d}$: $\Gamma = \{(u_1,
\dots , u_d)\in [0,1] \times \bb T^{d-1} : u_1 = \pm 1\}$.

For an open subset $\Lambda$ of $\bb R\times\bb T^{d-1}$, $\mc C^m
(\Lambda)$, $1\leq m\leq +\infty$, stands for the space of
$m$-continuously differentiable real functions defined on $\Lambda$.
Let $\mc C^m_0 (\Lambda)$ (resp. $\mc C^m_c (\Lambda)$), $1\leq m\leq
+\infty$, be the subset of functions in $\mc C^m (\Lambda)$ which
vanish at the boundary of $\Lambda$ (resp. with compact support in
$\Lambda$).

For an integer $N\ge 1$, denote by $\bb T_N^{d-1}=\{0,\dots,
N-1\}^{d-1}$, the discrete $(d-1)$-dimensional torus of length $N$.
Let ${D_N^d}=\{1,\ldots,N-1\} \times \bb T_N^{d-1}$ be the
cylinder in $\bb Z^d$ of length $N-1$ and basis $\bb T_N^{d-1}$ and let
$\G_N=\{(x_1, \dots, x_{d}) \in \bb Z\times \bb T_N^{d-1}\,|\, x_1 = 1 \hbox{~or~} x_1=(N-1)\}$ be the boundary of ${D_N^d}$. 

Let $\mathcal{V}\subset \mathbb{R}^d$ be a finite set of velocities $v = (v_1,\ldots,v_d)$. Assume that $\mathcal{V}$ is invariant under reflexions and permutations of the coordinates:
$$(v_1,\ldots,v_{i-1},-v_i,v_{i+1},\ldots,v_d)\hbox{~and~}(v_{\sigma(1)},\ldots,v_{\sigma(d)})$$
belong to $\mathcal{V}$ for all $1\leq i\leq d$, and all permutations $\sigma$ of $\{1,\ldots,d\}$, provided $(v_1,\ldots,v_d)$ belongs to $\mathcal{ V}$.

On each site of $D_N^d$, at most one particle for each velocity is allowed. We denote: the number of particles with velocity $v$ at $x$, $v\in \mathcal{ V}$, $x\in D_N^d$, by $\eta(x,v)\in \{0,1\}$;
the number of particles in each velocity $v$ at a site $x$ by $\eta_x = \{\eta(x,v); v\in\mathcal{ V}\}$; and a configuration by $\eta = \{\eta_x; x\in D_N^d\}$. The set of particle configurations is $X_N = \left( \{0,1\}^\mathcal{V}\right)^{D_N^d}$.

On the interior of the domain, the dynamics consists of two parts: (i) each particle of the system evolves according to a nearest neighbor weakly asymmetric random walk with exclusion among particles of the same velocity, and (ii) binary collision between particles of different velocities. Let $p(x,v)$ be an irreducible probability transition function of finite range, and mean velocity $v$:
$$\sum_x x p(x,v) = v.$$
The jump law and the waiting times are chosen so that the jump rate from site $x$ to site $x+y$ for a particle with velocity $v$ is
$$P_N(y,v) = \frac{1}{2}\sum_{j=1}^d (\delta_{y,e_j} + \delta_{y,-e_j}) + \frac{1}{N}p(y,v),$$
where $\delta_{x,y}$ stands for the Kronecker delta, which equals one if $x=y$ and 0 otherwise, and $\{e_1,\ldots,e_d\}$ is the canonical basis in $\mathbb{R}^d$.

\subsection{The boundary driven exclusion process}
Our main interest is to examine the stochastic lattice gas model given by the
generator ${\mathcal{L}}_N$ which is
the superposition of the boundary dynamics with the collision and exclusion:
\begin{equation}\label{geradorbd}
{\mathcal{ L}}_N = N^2\{\mathcal{L}_N^b + \mathcal{ L}_N^c +\mathcal{ L}_N^{ex}\},
\end{equation}
where $\mathcal{L}_N^b$ stands for the generator which models the part of the dynamics at which a particle at the boundary can enter or leave the system, $\mathcal{ L}_N^c$ stands for the generator which models the collision part of the dynamics and lastly, $\mathcal{ L}_N^{ex}$ models the exclusion part of the dynamics. Let $f$ be a local function on $X_N$.
The generator of the exclusion part of the dynamics, $\mathcal{L}_N^{ex}$, is given by
$$(\mathcal{ L}_N^{ex} f)(\eta) = \sum_{v\in\mathcal{ V}} \sum_{x,x+z\in D_N^d} \eta(x,v)[1-\eta(z,v)]P_N(z-x,v)\left[f(\eta^{x,z,v})-f(\eta) \right],$$
where
$$\eta^{x,y,v}(z,w) = \left\{
\begin{array}{cc}
\eta(y,v)&\text{if $w=v$ and $z=x$},\\
\eta(x,v)&\text{if $w=v$ and $z=y$},\\
\eta(z,w)&\text{otherwise}.
\end{array}
\right.$$

The generator of the collision part of the dynamics, $\mathcal{ L}_N^c$, is given by
$$(\mathcal{ L}_N^c f)(\eta) = \sum_{y\in D_N^d}\sum_{q\in\mathcal{ Q}} p(y,q,\eta)\left[ f(\eta^{y,q}) - f(\eta)\right],$$
where $\mathcal{Q}$ is the set of all collisions which preserve momentum:
$$\mathcal{Q} = \{q=(v,w,v',w') \in \mathcal{ V}^4 : v+w=v'+w'\},$$
the rate $p(y,q,\eta)$ is given by
$$p(y,q,\eta) = \eta(y,v)\eta(y,w)[1-\eta(y,v')][1-\eta(y,w')],$$
and for $q = (v_0,v_1,v_2,v_3)$, the configuration $\eta^{y,q}$ after the collision is defined as
$$\eta^{y,q}(z,u) = \left\{
\begin{array}{cc}
\eta(y,v_{j+2})&\text{if $z=y$ and $u=v_j$ for some $0\leq j\leq 3$},\\
\eta(z,u)&\text{otherwise,}
\end{array}
\right.$$
where the index of $v_{j+2}$ should be taken modulo 4. 

Particles of velocities $v$ and $w$ at the same site collide at rate one and produce two particles of velocities $v'$ and $w'$ at that site.

Finally, the generator of the boundary part of the dynamics is given by
\begin{eqnarray*}
(\mathcal{ L}_N^b f)(\eta) &=&\!\!\! \sum_{\substack{x\in D_N^d\\x_1 = 1}} \sum_{v\in\mathcal{ V}} [ \alpha_v(\tilde{x}/N)[1-\eta(x,v)] + (1-\alpha_v(\tilde{x}/N))\eta(x,v)][f(\sigma^{x,v}\eta)-f(\eta)]\\
&+&\!\!\! \sum_{\substack{x\in D_N^d\\x_1 = N-1}} \sum_{v\in\mathcal{ V}} [ \beta_v(\tilde{x}/N)[1-\eta(x,v)] + (1-\beta_v(\tilde{x}/N))\eta(x,v)][f(\sigma^{x,v}\eta)-f(\eta)],
\end{eqnarray*} 
where $\tilde{x} = (x_2,\ldots,x_d)$, 
$$\sigma^{x,v}\eta(y,w) = \left\{\begin{array}{cc}
1-\eta(x,w),& \hbox{if ~} w = v\hbox{~and~} y=x,\\
\eta(y,w),& \hbox{otherwise.}
\end{array}
\right.,$$
and for every $v\in\mathcal{V}$, $\alpha_v,\beta_v\in C^2(\mathbb{T}^{d-1})$. 
Note that time has been speeded up diffusively in \eqref{geradorbd}. We also assume that, for every $v\in\mathcal{V}$, $\alpha_v$ and $\beta_v$ have images belonging to some compact subset of $(0,1)$. The functions $\alpha_v$ and $\beta_v$, which
affect the birth and death rates at the two boundaries, represent the densities of the reservoirs.

Let $D(\mathbb{R}_+,X_N)$ be the set of right continuous functions with left limits taking values on $X_N$. For a probability measure $\mu$ on $X_N$, denote by $\mathbb{ P}_\mu$ the measure on the path space $D(\mathbb{ R}_{+}, X_N)$ induced by $\{\eta(t): t\geq 0\}$ and the initial measure $\mu$. Expectation with respect to $\mathbb{ P}_\mu$ is denoted by $\mathbb{ E}_\mu$. 
\subsection{Mass and momentum}\label{massmomentum}
For each configuration $\xi \in\{0,1\}^\mathcal{ V}$, denote by $I_0(\xi)$ the mass of $\xi$ and by $I_k(\xi)$, $k=1,\ldots,d,$ the momentum of $\xi$:
$$I_0(\xi) = \sum_{v\in\mathcal{ V}} \xi(v),\quad I_k(\xi)=\sum_{v\in\mathcal{ V}} v_k \xi(v).$$
Set ${\boldsymbol I}(\xi) := (I_0(\xi),\ldots,I_d(\xi))$. Assume that the set of velocities is chosen in such a way that the unique quantities
conserved by the random walk dynamics described above are mass and momentum: $\sum_{x\in D_N^d} {\boldsymbol I}(\eta_x)$. Two examples of sets of velocities satisfying these conditions can be found at \cite{emy}.

For each chemical potential ${\boldsymbol \lambda} = (\lambda_0,\ldots,\lambda_d) \in \mathbb{ R}^{d+1}$, denote by $m_\mathbb{ \lambda}$ the measure on $\{0,1\}^\mathcal{ V}$ given by
\begin{equation}\label{mprod}
m_\mathbb{ \lambda} (\xi) = \frac{1}{Z({\boldsymbol \lambda})} \exp\left\{\mathbb{ \lambda}\cdot {\boldsymbol I}(\xi)\right\},
\end{equation}
where $Z({\boldsymbol \lambda})$ is a normalizing constant. Note that $m_{\boldsymbol \lambda}$ is a product measure on $\{0,1\}^\mathcal{ V}$, i.e., that the variables $\{\xi(v): v\in\mathcal{ V}\}$ are independent under $m_{\boldsymbol \lambda}$.

Denote by $\mu_{\boldsymbol \lambda}^N$ the product measure on $X_N$, with marginals given by 
\begin{equation*}
\mu_{\boldsymbol \lambda}^N \{\eta: \eta(x,\cdot) = \xi\} = m_{\boldsymbol \lambda}(\xi),
\end{equation*}
for each $\xi$ in $\{0,1\}^\mathcal{ V}$ and $x\in D_N^d$. Note that $\{\eta(x,v): x\in D_N^d, v\in\mathcal{ V}\}$ are independent variables under $\mu_{\boldsymbol \lambda}^N$,
and that the measure $\mu_{\boldsymbol \lambda}^N$ is invariant for the exclusion process with periodic boundary condition.

The expectation under $\mu_{\boldsymbol \lambda}^N$ of the mass and momentum are given by
\begin{eqnarray*}
\rho({\boldsymbol \lambda})&:=& E_{\mu_{\boldsymbol \lambda}^N} \left[ I_0(\eta_x)\right] = \sum_{v\in\mathcal{ V}} \theta_v({\boldsymbol \lambda}),\\
p_k({\boldsymbol \lambda}) &:=& E_{\mu_{\boldsymbol \lambda}^N} \left[ I_k(\eta_x)\right] = \sum_{v\in\mathcal{ V}} v_k \theta_v({\boldsymbol \lambda}).
\end{eqnarray*}
In this formula $\theta_v({\boldsymbol \lambda})$ denotes the expected value of the density of particles with velocity $v$ under $m_{\boldsymbol \lambda}$:
$$\theta_v({\boldsymbol \lambda}):= E_{m_{\boldsymbol \lambda}} \left[\xi(v)\right] = \frac{\exp\left\{\lambda_0 + \sum_{k=1}^d \lambda_k v_k\right\}}{1+ \exp\left\{\lambda_0 + \sum_{k=1}^d \lambda_k v_k\right\}}.$$

Denote by $(\rho,{\boldsymbol p})({\boldsymbol \lambda}):= (\rho({\boldsymbol \lambda}),p_1({\boldsymbol \lambda}),\ldots, p_d({\boldsymbol \lambda}))$ the map that associates the chemical potential to the vector of density and momentum. It is possible to prove that $(\rho,{\boldsymbol p})$ is a diffeomorphism onto ${\mf U}\subset \mathbb{ R}^{d+1}$, the interior of the convex envelope of $\left\{ {\boldsymbol I}(\xi), \xi\in\{0,1\}^\mathcal{ V}\right\}$. Denote by $\Lambda = (\Lambda_0,\ldots,\Lambda_d): {\mf U}\to\mathbb{ R}^{d+1}$ the inverse of $(\rho, {\boldsymbol p})$. This correspondence allows one to parameterize the invariant states by the density and momentum: for each $(\rho,{\boldsymbol p})$ in ${\mf U}$ we have a product measure $\nu_{\rho,{\boldsymbol p}}^N = \mu_{\Lambda(\rho,{\boldsymbol p})}^N$ on $X_N$.

\subsection{Dynamical large deviations}

Fix $T>0$, let $\mathcal{ M}_{+}$
be the space of finite positive measures on $D^d$ endowed with the weak
topology, and let $\mathcal{ M}$ be the space of bounded variation signed measures on $D^d$
endowed with the weak topology. Let $\mathcal{ M}_{+}\times\mathcal{ M}^{d}$ be the cartesian
product of these spaces endowed with the product topology, which
is metrizable. Let also $\mathcal{M}^0$ be the subset of $\mathcal{M}_{+}\times\mc M^d$ of all
absolutely continuous measures with respect to the Lebesgue measure satisfying:

\begin{eqnarray*}
\mc{M}^0=\big\{&\pi&\in\mc{M}_{+}\times\mc M^d:\pi(du)=(\rho,\bs p)(u)du \;\; \hbox{ and }\\
&0&\leq \rho(u)\leq |\mc V|\;, |p_k(u)|\leq \breve{v}|\mc V|, k=1,\ldots,d,  \hbox{ a.e.} \big\}\, , 
\end{eqnarray*}
where $\breve{v} = \max_{v\in\mc V}v_1$. Let $D([0,T],\mathcal{ M}_{+}\times \mathcal{ M}^{d})$ be the set of right continuous functions with left limits taking values on $\mathcal{ M}_{+}\times\mathcal{ M}^{d}$ endowed with the Skorohod topology. $\mathcal{M}^0$ is a closed subset of $\mathcal{M}_{+}\times\mc M^d$ and $D([0,T],\mc M^0)$ is a closed subset of $D([0,T],\mc M_{+}\times \mc M^d)$.

For a measure $\pi\in\mc M$, denote by $\<\pi,G\>$ the integral of $G$ with respect to $\pi$.

Let $\Omega_T = (0,T)\times{D^d}$ and $\overline{\Omega_T} =
[0,T]\times\overline{{D^d}}$. For $1\leq m,n\leq +\infty$, denote
by $\mc C^{m,n}(\overline{\Omega_T})$ the space of functions $G =
G_t(u): \overline{\Omega_T}\to \bb R$ with $m$ continuous derivatives
in time and $n$ continuous derivatives in space. We also denote by
$\mc C^{m,n}_0(\overline{\Omega_T})$ (resp. $\mc
C^{\infty}_c(\Omega_T)$) the set of functions in $\mc
C^{m,n}(\overline{\Omega_T})$ (resp. $\mc
C^{\infty,\infty}(\overline{\Omega_T})$) which vanish at
$[0,T]\times\Gamma$ (resp. with compact support in $\Omega_T$).

Let the energy $\mc Q:D([0,T],\mc M^0)\to[0,\infty]$ be given by
\begin{eqnarray*}
\mc Q(\pi) = \sum_{k=0}^d\sum_{i=1}^d \sup_{G\in\mc C^{\infty}_c(\Omega_T)}
\Big\{ 2 \int_0^T dt\; \langle p_{k,t},\partial_{u_i}G_t\rangle 
- \int_0^Tdt\int_{{D^d}} G(t,u)^2\, du \Big\}\, . 
\end{eqnarray*}
where $p_{k,t}(u) = p_k(t,u)$ and $p_{0,t}(u) = \rho(t,u)$.

For each $G\in\mc C_0^{1,2}(\overline{\Omega_T})\times[\mc C_0^2(\overline{D^d})]^d$ and each measurable
function $\gamma:\overline{D^d}\to[0,|\mc V|]\times[-\breve{v}|\mc V|,\breve{v}|\mc V|]^d$, $\gamma=(\rho_0,\bs p_0)$, let $\hat J_G = \hat
J_{G,\gamma,T}:D([0,T],\mc M^0)\to\bb R$ be the functional given by
\begin{eqnarray*}
\hat J_{G}(\pi) & = & \int_{{D^d}} G(T,u)\cdot ({\rho},{{\boldsymbol p}})(T,u)du - \int_{{D^d}} G(0,u)\cdot(\rho_0, {\boldsymbol p}_0) (u) du\\
&-&\int_0^T dt \int_{{D^d}} du \left\{ ({\rho},{{\boldsymbol p}})(t,u)\cdot\partial_t G(t,u) + \frac{1}{2} ({\rho},{{\boldsymbol p}})(t,u) \cdot \sum_{1\leq i\leq d} \partial_{u_i}^2 G(t,u)\right\}\\
&+&\int_0^T dt \int_{\{1\}\times\mathbb{ T}^{d-1}}\!\!\! dS\,\, b(\tilde{u})\cdot \partial_{u_1} G(t,u) - \int_0^T dt \int_{\{0\}\times\mathbb{ T}^{d-1}}\!\!\! dS\,\, a(\tilde{u})\cdot \partial_{u_1} G(t,u)\\
&+&\int_0^T dt \int_{{D^d}} du\,\,  \sum_{v\in\mathcal{ V}} \tilde{v}\cdot \chi(\theta_v(\Lambda({\rho},{{\boldsymbol p}}))) \sum_{1\leq i \leq d} v_i \partial_{u_i} G(t,u)\\
&-& \int_0^T dt\int_{{D^d}}du \sum_{v\in\mc V} \left(\sum_{k=0}^d v_k \partial_{x_i}G_t^k(u) \right)^2\chi(\theta_v(\Lambda(\rho,p))),
\end{eqnarray*}
where $\chi(r)=r(1-r)$ is the static compressibility and $\pi_t(du) =
(\rho,\bs p)(t,u) du$. Define $J_G = J_{G,\gamma,T}:D([0,T],\mc M_{+}\times\mc M^d)\to\bb R$ by
\begin{equation*}
J_G (\pi) =
\begin{cases}
\displaystyle \hat J_G (\pi) & \hbox{ if }  \pi \in D([0,T],\mc M^0) ,\\ 
+\infty & \hbox{ otherwise .}
\end{cases}
\end{equation*}

We define the rate functional
$I_T(\cdot|\gamma):D([0,T],\mc M_{+}\times\mc M^d)\to[0,+\infty]$ as
\begin{equation*}
I_T(\pi|\gamma) =
\begin{cases}
\displaystyle \sup_{G\in\cc^{1,2}_0(\overline{\Omega_T})\times[\mc C_0^2(\overline{D^d})]^d}
\!\big\{J_G(\pi)\big\} & \hbox{ if } \mc Q(\pi)<\infty\, ,\\ 
+\infty & \hbox{ otherwise .}
\end{cases}
\end{equation*}

We now present the main result of this article, whose proof is given in Section \ref{sec6}, which is the dynamical large deviations for this boundary driven exclusion process with many conserved quantities.

\begin{theorem} 
\label{mt} 
Fix $T>0$ and a measurable function $(\rho_0,\bs p_0):{D^d}\to[0,|\mc V|]\times[-\breve{v}|\mc V|,\breve{v}|\mc V|]^d$. Consider
a sequence $\eta^N$ of configurations in $X_N$ associated to $\gamma=(\rho_0,\bs p_0)$
in the sense that:
\begin{equation*}
\lim_{N\to\infty} \<\pi_0^N (\eta^N) , G\> \; =\;
\int_{D^d} G(u) \rho_0(u) \, du, 
\end{equation*}
and
\begin{equation*}
\lim_{N\to\infty} \<\pi_k^N (\eta^N) , G\> \; =\;
\int_{D^d} G(u) p_k(u) \, du, \quad k=1,\ldots,d,
\end{equation*}
for every continuous function $G:\overline{{D^d}}\to\bb R$. Then, the
measure $Q_{\eta^N}=\bb P_{\eta^N}(\pi^N)^{-1}$ on $D([0,T],\cm_{+}\times \mc M^d)$
satisfies a large deviation principle with speed $N^d$ and rate
function $I_T(\cdot|\gamma)$. Namely, for each closed set $\cc\subset
D([0,T],\cm_{+}\times\mc M^d)$,
\begin{equation*}
\limsup_{N\to\infty}\frac{1}{N^d}\log
Q_{\eta^N}(\cc)\leq - \inf_{\pi\in\cc} I_T(\pi|\gamma)
\end{equation*}
and for each open set $\mc{O}\subset D([0,T],\cm_{+}\times\mc M^d)$,
\begin{equation*}
\liminf_{N\to\infty}\frac{1}{N^d}\log Q_{\eta^N}(\mc{O})\geq -
\inf_{\pi\in\mc{O}} I_T(\pi|\gamma)\;.
\end{equation*}
Moreover, the rate function $I_T(\cdot|\gamma)$ is lower
semicontinuous and has compact level sets.
\end{theorem}

\section{Hydrodynamics}
\label{sec3}

\renewcommand{\labelenumi}{({\bf H\theenumi})}

Fix $T>0$ and let $(B,\|\cdot\|_B)$ be a Banach space. We denote by $L^2([0,T],B)$ the Banach space of measurable functions $U:[0,T]\to B$ for which
$$\|U\|_{L^2([0,T],B)}^2 = \int_0^T \|U_t\|_B^2 dt <\infty.$$
Moreover, we denote by $H^1({D^d})$ the Sobolev space of measurable functions in $L^2({D^d})$ that have generalized derivatives in $L^2({D^d})$.

For $x = (x_1,\tilde{x})\in \{0,1\}\times\mathbb{T}^{d-1}$, let
\begin{equation}\label{funcaod}
d(x) =\left\{
\begin{array}{cc}
a(\tilde{x}) = \sum_{v\in\mathcal{V}} (\alpha_v(\tilde{x}),v_1\alpha_v(\tilde{x}),\ldots,v_d\alpha_v(\tilde{x})),&\hbox{if~~}x_1=0,\\[10pt] 
b(\tilde{x}) = \sum_{v\in\mathcal{V}} (\beta_v(\tilde{x}),v_1\beta_v(\tilde{x}),\ldots,v_d\beta_v(\tilde{x})),&\hbox{if~~}x_1=1.
\end{array}
\right.
\end{equation}

Fix a bounded density profile $\rho_0:{D^d} \to \mathbb{ R}_{+}$, and a bounded momentum
profile ${\boldsymbol p}_0: {D^d} \to \mathbb{ R}^d$.
A bounded function $({\rho},{{\boldsymbol p}}): [0,T]\times {D^d} \to \mathbb{ R}_{+}\times \mathbb{ R}^d$ is a weak solution
of the system of parabolic partial differential equations
\begin{equation}\label{problema2}
\left\{ 
\begin{array}{c}
\partial_t (\rho,{\boldsymbol p}) + \sum_{v\in \mathcal{ V}} \tilde{v}\left[v\cdot \nabla\chi(\theta_v(\Lambda(\rho,{\boldsymbol p})))\right] = \frac{1}{2}\Delta (\rho,{\boldsymbol p}),\\[10pt]
(\rho,{\boldsymbol p})(0,\cdot) =  (\rho_0,{\boldsymbol p}_0)(\cdot) \hbox{~and~} (\rho,{\boldsymbol p})(t,x) = d(x), x\in \{0,1\}\times\mathbb{T}^{d-1},
\end{array}
\right.
\end{equation}
if for every vector valued function $H:[0,T]\times {D^d}\to\mathbb{ R}^{d+1}$ of class $C^{1,2}\left([0,T]\times {D^d}\right)$
vanishing at the boundary, we have
$$
\int_{{D^d}} H(T,u)\cdot ({\rho},{{\boldsymbol p}})(T,u)du - \int_{{D^d}} H(0,u)\cdot(\rho_0, {\boldsymbol p}_0) (u) du
$$
$$
=\int_0^T dt \int_{{D^d}} du \left\{ ({\rho},{{\boldsymbol p}})(t,u)\cdot\partial_t H(t,u) + \frac{1}{2} ({\rho},{{\boldsymbol p}})(t,u) \cdot \sum_{1\leq i\leq d} \partial_{u_i}^2 H(t,u)\right\}
$$
$$
-\int_0^T dt \int_{\{1\}\times\mathbb{ T}^{d-1}} dS\,\, b(\tilde{u})\cdot \partial_{u_1} H(t,u) + \int_0^T dt \int_{\{0\}\times\mathbb{ T}^{d-1}} dS\,\, a(\tilde{u})\cdot \partial_{u_1} H(t,u)
$$
$$
-\int_0^T dt \int_{{D^d}} du\,\,  \sum_{v\in\mathcal{ V}} \tilde{v}\cdot \chi(\theta_v(\Lambda({\rho},{{\boldsymbol p}}))) \sum_{1\leq i \leq d} v_i \partial_{u_i} H(t,u),
$$
$dS$ being the Lebesgue measure on $\mathbb{T}^{d-1}$. 

We say that that the solution $(\rho,{\boldsymbol p})$ has finite energy if its components belong to $L^2([0,T],H^1({D^d}))$:
$$\int_0^T ds \left(\int_{{D^d}} \|\nabla \rho(s,u)\|^2 du\right)<\infty,$$
and 
$$\int_0^T ds \left(\int_{{D^d}} \|\nabla p_k(s,u)\|^2 du\right)<\infty,$$
for $k=1,\ldots,d$, where $\nabla f$ represents the generalized gradient of the function $f$.

In \cite{s} the following theorem was proved:

\begin{theorem}\label{limhyd2} Let $(\mu^N)_N$ be a sequence of probability measures on $X_N$ associated to the profile $(\rho_0,{\boldsymbol p}_0)$. Then, for every $t\geq 0$, for every continuous function $H:{D^d}\to \mathbb{ R}$ vanishing at the boundary, and for every $\delta > 0$,
$$\lim_{N\to\infty} \mathbb{ P}_{\mu^N} \left[ \left| \frac{1}{N^d} \sum_{x\in D_N^d} H\left(\frac{x}{N}\right) I_0(\eta_x(t)) - \int_{{D^d}} H(u) {\rho}(t,u)du\right|>\delta\right] = 0,$$
and for $1\leq k \leq d$
$$\lim_{N\to\infty} \mathbb{ P}_{\mu^N} \left[ \left| \frac{1}{N^d} \sum_{x\in D_N^d} H\left(\frac{x}{N}\right) I_k(\eta_x(t)) - \int_{{D^d}} H(u) {p_k}(t,u)du\right|>\delta\right] = 0,$$
where $({\rho},{\boldsymbol p})$ has finite energy and is the unique weak solution of equation (\ref{problema2}).
\end{theorem}

\section{The rate function $I_T(\cdot | \gamma)$}
\label{sec4}


We examine in this section the rate function $I_T(\cdot|\gamma)$. The
main result, presented in Theorem \ref{th4} below, states that
$I_T(\cdot | \gamma)$ has compact level sets. The proof relies on two
ingredients. The first one, stated in Lemma \ref{lem03}, is an
estimate of the energy and of the $H_{-1}$ norm of the time derivative
of a trajectory in terms of the rate function. The second one, stated
in Lemma \ref{lem02}, establishes that sequences of trajectories, with
rate function uniformly bounded, which converges weakly in $L^2$
converge in fact strongly. We follow the strategy introduced in \cite{flm}.

Recall that $V$ is an open neighborhood of $D^d$, and consider, for each $v\in\mc V$, smooth functions $\kappa_k^v:V\to(0,1)$ in $C^2(V)$, for $k=0,\ldots,d$. We assume that each $\kappa_k^v$ has its image contained in some compact subset of $(0,1)$, that the restriction of $\kappa = \sum_{v\in\mc V}(\kappa_0^v,v_1\kappa_1^v,\ldots,v_d\kappa_d^v)$ to $\{0\}\times\mathbb{T}^{d-1}$ equals the vector valued function $a(\cdot)$ defined in \eqref{funcaod}, and that the restriction of $\kappa$ to $\{1\}\times\mathbb{T}^{d-1}$ equals the vector valued function $b(\cdot)$, also defined in \eqref{funcaod}, in the sense that $\kappa(x) = d(x_1,\tilde{x})$ if $x\in\{0,1\}\times\mathbb{T}^{d-1}$.

Let $L^2(D^d)$ be the Hilbert space of functions $G:D^d
\to \bb R$ such that $\int_{D^d} | G(u) |^2 du <\infty$ equipped with
the inner product
\begin{equation*}
\<G,F\>_2 =\int_\Omega G(u) \, F (u) \, du\; ,
\end{equation*}
and the norm of $L^2(D^d)$ is denoted by $\| \cdot \|_2$.

Recall that $H^1(D^d)$ is the Sobolev space of functions $G$ with
generalized derivatives $\partial_{u_1} G, \dots , \partial_{u_d} G$
in $L^2(D^d)$. $H^1(D^d)$ endowed with the scalar product
$\<\cdot, \cdot\>_{1,2}$, defined by
\begin{equation*}
\<G,F\>_{1,2} = \< G, F \>_2 + \sum_{j=1}^d
\<\partial_{u_j} G \, , \, \partial_{u_j} F \>_2\;,
\end{equation*}
is a Hilbert space. The corresponding norm is denoted by
$\|\cdot\|_{1,2}$.

Recall that we denote by
${\mathcal C}_{c}^\infty ({D^d})$ the set of infinitely
differentiable functions $G:{D^d} \to \R$, with compact support in
${D^d}$. Denote by
$H^1_0({D^d})$ the closure of $C_c^{\infty}({D^d})$ in
$H^1({D^d})$. Since ${D^d}$ is bounded, by Poincar\'e's inequality,
there exists a finite constant $C$ such that for all $G\in
H^1_0({D^d})$
\begin{equation*}
\|G\|^2_2 \;\le\;  C \sum_{j=1}^d \<\partial_{u_j} G \, , \, 
\partial_{u_j} G \>_2 \; .
\end{equation*}
This implies that, in $H^1_0 ({D^d})$
\begin{equation*}
\|G\|_{1,2,0} \;=\; \Big\{ \sum_{j=1}^d
\<\partial_{u_j} G \, , \, \partial_{u_j} G \>_2  \Big\}^{1/2}
\end{equation*}
is a norm equivalent to the norm $\|\cdot \|_{1,2}$.  Moreover, $H^1_0
({D^d})$ is a Hilbert space with inner product given by
\begin{equation*}
\< G \, , \, J \>_{1,2,0}
\;=\; \sum_{j=1}^d
\<\partial_{u_j} G \, , \, \partial_{u_j} J \>_2 \; .
\end{equation*}

To assign boundary values along the boundary $\Gamma$ of ${D^d}$ to
any function $G$ in $H^1({D^d})$, recall, from the trace Theorem
(\cite{z}, Theorem 21.A.(e)), that there exists a continuous linear
operator $\Tr:H^1({D^d})\to L^2(\Gamma)$, called trace, such that $\Tr (G) =
G\big|_{\Gamma}$ if $G\in H^1({D^d})\cap \mc C(\overline{{D^d}})$.
Moreover, the space $H^1_0({D^d})$ is the space of functions $G$ in
$H^1({D^d})$ with zero trace (\cite{z}, Appendix (48b)):
\begin{equation*}
H^1_0({D^d}) = \left\{G\in H^1({D^d}):\; \Tr (G) = 0\right\}\,.
\end{equation*}

Finally, denote by $H^{-1}({D^d})$ the dual of $H^1_0({D^d})$.
$H^{-1}({D^d})$ is a Banach space with norm $\Vert\cdot\Vert_{-1}$
given by
\begin{equation*}
\Vert v\Vert^2_{-1} = \sup_{G\in\mc C^{\infty}_c({D^d})}
\left\{2\langle v,G\rangle_{-1,1} -
\int_{{D^d}} \Vert \nabla G(u)\Vert^2du \right\}\, , 
\end{equation*}
where $\langle v,G\rangle_{-1,1}$ stands for the values of the linear
form $v$ at $G$.

For each $G\in\mc C^{\infty}_c(\Omega_T)$ and each integer $1\leq
i\leq d$, let $\mc Q_{i,k}^G:D([0,T],\mc M^0)\to\bb R$ be the functional
given by 
\begin{eqnarray*}
\mc Q_{i,k}^G(\pi) = 2\int_0^Tdt\;\langle\pi_t^k,\partial_{u_i}G_t\rangle -
\int_0^Tdt\int_{{D^d}} du\; G(t,u)^2\, , 
\end{eqnarray*}
where $\pi = (\pi^0,\pi^1,\ldots,\pi^d)$. Recall, from subsection 2.2, that the energy $\mc Q(\pi)$ is given by
\begin{equation*}
\mc Q(\pi) = \sum_{k=0}^d\sum_{i=1}^d\mc Q_{i,k}(\pi),\;\; \hbox{ with } \;\;\mc
Q_{i,k}(\pi) = \sup_{G\in\mc C_c^{\infty}(\Omega_T)}{\mc Q_{i,k}^G(\pi)}\, . 
\end{equation*}

The functional $\mc Q^G_{i,k}$ is convex and continuous in the Skorohod
topology. Therefore $\mc Q_{i,k}$ and $\mc Q$ are convex and lower
semicontinuous. Furthermore, it is well known that a measure
$\pi(t,du) = (\rho,\bs p)(t,u) du$ in $D([0,T], \mc M_{+}\times\mc M^d)$ has finite energy,
$\mc Q(\pi) < \infty$, if and only if its density $\rho$ and its momentum $\bs p$ belong to
$L^2([0,T] , H^1({D^d}))$. In such case
\begin{equation*}
\hat{\mc Q}(\pi) \;:=\;
\sum_{k=0}^d\int_0^Tdt\int_{{D^d}}du\;\Vert\nabla p_{k,t}(u)\Vert^2 \;<\; \infty,
\end{equation*}
where $p_{0,t}(u) = \rho(t,u)$. We also have that $\mc Q(\pi) = \hat{\mc Q}(\pi)$.

Let $D_{\gamma} = D_{\gamma,b}$ be the subset of $C([0,T],\mc M^0)$
consisting of all paths $\pi(t,du) = (\rho,\bs p)(t,u) du$ with initial
profile $\gamma(\cdot) = (\rho_0,\bs p_0)(\cdot)$, finite energy $\mc Q(\pi)$ (in
which case $\rho_t$ and $\bs p_t$ belong to $H^1({D^d})$ for almost all $0\leq
t\leq T$ and so $\Tr(\rho_t)$ is well defined for those $t$) and such
that $\Tr(\rho_t) = d_0$ and $\Tr(p_{k,t})=d_k$, $k=1,\ldots,d$, for almost all $t$ in $[0,T]$, where $d(\cdot) = (d_0(\cdot),d_1(\cdot),\ldots,d_d(\cdot))$.

\begin{lemma}
\label{lem01}
Let $\pi$ be a trajectory in $D([0,T],\mc M_{+}\times\mc M^d)$ such that
$I_T(\pi|\gamma)<\infty$. Then $\pi$ belongs to $D_{\gamma}$.
\end{lemma}

\begin{proof}
Fix a path $\pi$ in $D([0,T],\mc M_{+}\times\mc M^d)$ with finite rate function,
$I_T(\pi|\gamma)<\infty$. By definition of $I_T$, $\pi$ belongs to
$D([0,T],\mc M^0)$. Denote its density and momentum by $(\rho,\bs p)$: $\pi(t,du) =
(\rho,\bs p)(t,u) du$. 

The proof that $(\rho,\bs p)(0,\cdot) = \gamma(\cdot)$ is similar to the one
of Lemma 3.5 in \cite{BDGJL}, and the proof that $\Tr(\rho_t) = d_0$, $\Tr(p_{k,t})=d_k$, $k=1,\ldots,d$, is similar to the one found in
Lemma 4.1 in \cite{flm}. The fact that $\pi$ has finite energy follows from Lemma \ref{lemmaenergiald}.

We deal now with the continuity of $\pi$. We claim that there exists a
positive constant $C_0$ such that, for any $g\in\mc
C^{\infty}_c({D^d})$, and any $0 \leq s < r < T$,
\begin{eqnarray}\label{cont}
|\langle\pi_r,g\rangle-\langle\pi_s,g\rangle| \;\leq\;
C_0(r-s)^{1/2}\left\{C_1+I_T(\pi|\gamma)+\Vert
  g\Vert^2_{1,2,0}+(r-s)^{1/2}\Vert\Delta g\Vert_1\right\}\, . 
\end{eqnarray}
Indeed, for each $\delta>0$, let $\psi^{\delta}:[0,T]\to\bb R$ be the
function given by
\begin{equation*}
(r-s)^{1/2}\psi^{\delta}(t) = 
\begin{cases}
0 & \hbox{ if } 0\leq t\leq s \;\hbox{ or }\; r+\delta\leq t\leq T\, , \\
\frac{t-s}{\delta} & \hbox{ if } s\leq t\leq s+\delta\, , \\
1 & \hbox{ if } s+\delta\leq t\leq r\, , \\
1-\frac{t-r}{\delta} & \hbox{ if } r\leq t\leq r+\delta\, ,
\end{cases}
\end{equation*}
and let $G^{\delta}(t,u) = \psi^{\delta}(t)g(u)$. Of course,
$G^{\delta}$ can be approximated by functions in $\mc
C^{1,2}_0(\overline{\Omega_T})$ and then
\begin{eqnarray*}
(r-s)^{1/2}\lim_{\delta\to 0}J_{G^{\delta}}(\pi) & = &
\langle\pi_r,g\rangle-\langle\pi_s,g\rangle -
\int_s^rdt\;\langle\pi_t,\Delta g\rangle \\
&+& \int_r^sdt \int_{{D^d}}du \sum_{v\in\mc V} \tilde{v}\cdot \chi(\theta_v(\Lambda(\rho,\bs p))) \sum_{i=1}^d v_i\partial_{u_i}g(u)\\
&-& \frac{1}{(r-s)^{1/2}}\int_s^r dt\int_{{D^d}}du \sum_{v\in\mc V}\left(\sum_{k=0}^d\partial_{x_i}v_k g^k(u)\right)^2\chi(\theta_v(\Lambda(\rho,\bs p)))
\end{eqnarray*}
To conclude the proof, we observe that the left-hand side
is bounded by $(r-s)^{1/2} I_T(\pi|\gamma)$, that
$\chi$ is positive and bounded above on $[0,1]$ by $1/4$, and finally, we use the elementary inequality $2ab\leq a^2+b^2$.
\end{proof}

Denote by $L^2([0,T],H_0^1({D^d}))^*$ the dual of
$L^2([0,T],H_0^1({D^d}))$.  By Proposition 23.7 in \cite{z},
$L^2([0,T],H_0^1({D^d}))^*$ corresponds to
$L^2([0,T],H^{-1}({D^d}))$ and for $v$ in
$L^2([0,T],H_0^1({D^d}))^*$, $G$ in $L^2([0,T],H_0^1({D^d}))$,
\begin{equation}
\label{fl1}
\<\!\< v,G \>\!\>_{-1,1} \; =\; \int_0^T \<v_t, G_t\>_{-1,1}\, dt \; , 
\end{equation}
where the left hand side stands for the value of the linear functional
$v$ at $G$. Moreover, if we denote by $|\!| \!| v |\!| \!|_{-1}$ the
norm of $v$,
\begin{equation*}
|\!| \!| v |\!| \!|^2_{-1} \;=\; \int_0^T \Vert v_t \Vert^2_{-1} \, dt\;.
\end{equation*}

Fix a path $\pi(t,du)=(\rho,\bs p)(t,u)du$ in $D_\gamma$ and suppose that for $k=0,\ldots,d$
\begin{equation}
\label{cl1}
\sup_{H\in\mc C^{\infty}_c(\Omega_T)}\Big \{2 \int_0^T
  dt\, \langle p_{k,t},\partial_tH_t\rangle_2 - \int_0^T dt\int_{{D^d}}
  du\; \Vert\nabla H_t\Vert^2\Big\}\;<\; \infty\;.
\end{equation} 
In this case, for each $k$, $\partial_t p_k : C^{\infty}_c(\Omega_T) \to \bb R$
defined by
\begin{equation*}
\partial_t p_k (H) \;=\; - \int_0^T \< p_{k,t}, \partial_t H_t\>_2\, dt
\end{equation*}
can be extended to a bounded linear operator $\partial_t p_k :
L^2([0,T],H_0^1({D^d})) \to \bb R$. It belongs therefore to
$L^2([0,T],H_0^1({D^d}))^* = L^2([0,T],H^{-1}({D^d}))$. In
particular, there exists $v^k = \{v_t^k :0\le t\le T\}$ in
$L^2([0,T],H^{-1}({D^d}))$, which we denote by $v_t^k = \partial_t
p_{k,t}$, such that for any $H$ in $L^2([0,T],H_0^1({D^d}))$,
\begin{equation*}
\<\!\< \partial_t p_k , H\>\!\>_{-1,1} \;=\;
\int_0^T \langle \partial_t p_{k,t} , H_t\rangle_{-1,1}\,dt \;.
\end{equation*}
Moreover,
\begin{equation*}
\begin{split}
|\!| \!| \partial_t p_k |\!| \!|^2_{-1} \; &=\;
\int_0^T \, \Vert \partial_t p_{k,t} \Vert^2_{-1} \, dt\\
&=\; \sup_{H\in\mc C^{\infty}_c(\Omega_T)}\Big \{2 \int_0^T
  dt\, \langle p_{k,t},\partial_tH_t\rangle_2 - \int_0^T dt\int_{{D^d}}
  du\; \Vert\nabla H_t\Vert^2\Big\} \;.
\end{split}
\end{equation*}

Denote by $\<\!\< \partial_t(\rho,\bs p) , G\>\!\>_{-1,1}$ the linear functional given by
$$\<\!\< \partial_t(\rho,\bs p) , G\>\!\>_{-1,1} = \sum_{k=0}^d \<\!\< \partial_t p_k , H\>\!\>_{-1,1},$$
with
$$|\!| \!| \partial_t (\rho,\bs p) |\!| \!|^2_{-1} = \sum_{k=0}^d |\!| \!| \partial_t p_k |\!| \!|^2_{-1}.$$

Let $W$ be the set of paths $\pi(t,du)=(\rho,\bs p)(t,u)du$ in $D_\gamma$ such
that \eqref{cl1} holds, i.e., such that $\partial_t p_k$ belongs to
$L^2\left([0,T], H^{-1}({D^d})\right)$.  For $G$ in
$L^2\left([0,T],[H^1_0({D^d})]^{d+1}\right)$, let $\bb J_{G}:W\to \bb R$ be
the functional given by
\begin{eqnarray*}
\bb J_{G}(\pi) & = & \<\!\< \partial_t(\rho,\bs p) , G\>\!\>_{-1,1} + \frac{1}{2}\int_0^T dt \int_{{D^d}} du \nabla({\rho},{{\boldsymbol p}})(t,u) \cdot \nabla G(t,u)\\
&+&\int_0^T dt \int_{{D^d}} du\,\,  \sum_{v\in\mathcal{ V}} \tilde{v}\cdot \chi(\theta_v(\Lambda({\rho},{{\boldsymbol p}}))) \sum_{1\leq i \leq d} v_i \partial_{u_i} G(t,u)\\
&-& \int_0^T dt\int_{{D^d}}du \sum_{v\in\mc V} \left(\sum_{k=0}^d v_k \partial_{x_i}G_t^k(u) \right)^2\chi(\theta_v(\Lambda(\rho,p))),
\end{eqnarray*}
Note that $\bb J_{G}(\pi) = J_G(\pi)$ for every $G$ in $C^{\infty}_c(\Omega_T)\times[\mc
C^{\infty}_c(D^d)]^{d}$. Moreover, since $\bb J_{\cdot} (\pi)$ is
continuous in $L^2\left([0,T],[H^1_0 ({D^d})]^{d+1}\right)$ and since $\mc
C^{\infty}_c(\Omega_T)$ is dense in $\mc C^{1,2}_0
(\overline{\Omega_T})$ and in $L^2 ([0,T]$, $H^1_0({D^d}))$, for
every $\pi$ in $W$,
\begin{eqnarray}
\label{rf1}
I_T(\pi|\gamma) = \sup_{G\in C^{\infty}_c\Omega_T\times[\mc C^{\infty}_c(D^d)]^{d}} 
\bb J_{G}(\pi) \;=\;
\sup_{G\in L^2\left([0,T],[H^1_0]^{d+1}\right)} \bb J_{G}(\pi)\, . 
\end{eqnarray}

\begin{lemma}
\label{lem03}
There exists a constant $C_0>0$ such that if the density and momentum $(\rho, \bs p)$ of
some path $\pi (t,du) = (\rho,\bs p)(t,u) du$ in $D([0,T],\mc M^0)$ has
generalized gradients, $\nabla\rho$ and $\nabla p_k$, $k=1,\ldots,d$. Then
\begin{eqnarray}
\label{est1}
|\!| \!| \partial_t (\rho,\bs p) |\!| \!|^2_{-1} &\leq&
C_0\left\{I_T(\pi|\gamma)+\mc Q(\pi)\right\}\, , \\
\label{est2}
\sum_{k=0}^d\int_0^T dt \int_{{D^d}}
du\;{\Vert\nabla p_k(t,u)\Vert^2} &\leq&
C_0\left\{I_T(\pi|\gamma)  +1\right\}\, .
\end{eqnarray}
\end{lemma}

\begin{proof}
Fix a path $\pi(t,du) = (\rho,\bs p)(t,u) du$ in $D([0,T],\mc M^0)$.  In view
of the discussion presented before the lemma, we need to show that the
left hand side of \eqref{cl1} is bounded by the right hand side of
\eqref{est1}.  Such an estimate follows from the definition of the
rate function $I_T(\cdot|\gamma)$ and from the elementary inequality
$2ab\leq Aa^2+A^{-1}b^2$.

To prove \eqref{est2}, observe that
\begin{eqnarray*}
I(\pi) &\geq& J_G(\pi) = \partial_t \pi(G) + \frac{1}{2} \int_0^Tdt \int_{D^d}du \sum_{i=1}^d\<\partial_{x_i}(\rho,p),\partial_{x_i}G\>_2\\
&+& \int_0^Tdt \int_{D^d}du\sum_{v\in\mc V} \left(\chi(\theta_v(\Lambda(\rho,p))) \right) \sum_{i=1}^d \tilde{v}(v_i\partial_{x_i}G)\\
&-& \int_0^Tdt \int_{D^d}du\sum_{v\in\mc V} \sum_{i=1}^d\left(\sum_{k=0}^d v_k\partial_{\rho_i}G^k \right)^2\chi(\theta_v(\Lambda(\rho,p)))\\
&\geq & \partial_t\pi(G) +\frac 12\int_0^Tdt \int_{D^d}du \sum_{i=1}^d\<\partial_{x_i}(\rho,p),\partial_{x_i}G\>_2 - C\int_0^Tdt\sum_{k=0}^d\|\nabla G^k\|_2^2,
\end{eqnarray*}
where $C$ is constant obtained from the elementary inequality $2ab\leq a^2+b^2$, the fact that $\mc V$ is finite, and that $\chi$ is bounded above by $1/4$ in $[0,1]$.

Now, consider $G = K(\pi - \kappa)$, and note that $\pi-\kappa$ belong to $L^2([0,T],H_0^1({D^d}))$, which implies that it may be approximated by $C_c^\infty$ functions. Therefore $\partial_t\pi(G) = \<\pi_T,\pi_T-\kappa\>-\<\pi_0,\pi_0-\kappa\>$, which is bounded by some constant $C_1$. We, then, obtain that
\begin{eqnarray*}
I(\pi)&\ge& \int_0^Tdt\Big\{ -C_1 +\frac K2\sum_{k=0}^d\|\nabla p_k\|_2^2 - \frac K2 \sum_{i=1}^d\<\partial_{x_i}(\rho,p),\partial_{x_i}\kappa\>_2 - CK^2\sum_{k=0}^d\|\nabla(p_k-\kappa_k)\|_2^2\Big\}\\
&\ge& \int_0^Tdt\Big\{\Big (K/4-2CK^2\Big)\sum_{k=0}^d\|\nabla p_k\|_2^2\Big\}- \frac K4\sum_{k=0}^d\|\nabla\kappa_k\|_2^2 - 2CK^2\sum_{k=0}^d\|\nabla\kappa_k\|_2^2 -C_1
\end{eqnarray*}
where in the last inequality we used the Cauchy-Schwartz inequality and the elementary inequalities $2ab\leq a^2+b^2$. The proof thus follows from choosing a suitable $K$, the estimate given in \eqref{est1}, and the fact we have a fixed smooth function $\kappa$.
\end{proof}

\begin{corollary}
\label{rfhe}
The density $(\rho,\bs p)$ of a path $\pi(t,du)=(\rho,\bs p)(t,u)du$ in $D([0,T],\mc
M^0)$ is the weak solution of the equation \eqref{problema2} and initial
profile $\gamma$ if and only if the rate function $I_T(\pi|\gamma)$
vanishes. Moreover, if any of the above conditions hold, $\pi$ has finite energy ($\mc Q(\pi)<\infty$).
\end{corollary}

\begin{proof}
  On the one hand, if the density $(\rho,\bs p)$ of a path $\pi(t,du)=
  (\rho,\bs p)(t,u)du$ in $D([0,T], \mc M^0)$ is the weak solution of equation
  \eqref{problema2} with initial condition is $\gamma$, in the
  formula of $\hat J_G(\pi)$, the linear part in $G$ vanishes which
  proves that the rate functional $I_T(\pi|\gamma)$ vanishes. On the
  other hand, if the rate functional vanishes, the path $(\rho,\bs p)$ belongs
  to $L^2([0,T],[H^1({D^d})]^{d+1})$ and the linear part in $G$ of $J_G(\pi)$
  has to vanish for all functions $G$. In particular, $(\rho,\bs p)$ is a weak
  solution of \eqref{problema2}. Moreover, if the rate function is finite, by the previous lemma, 
  $\pi$ has finite energy. Accordingly, if $\pi$ is a weak solution, we have from Theorem \ref{limhyd2} that
  it has finite energy.
\end{proof}

For each $q>0$, let $E_q$ be the level set of $I_T(\pi|\gamma)$
defined by
\begin{equation*}
E_q=\left\{\pi\in D([0,T],\mc M): I_T(\pi|\gamma)\leq q\right\}\, .
\end{equation*}
By Lemma \ref{lem01}, $E_q$ is a subset of $C([0,T],\mc M^0)$. Thus,
from the previous lemma, it is easy to deduce the next result.

\begin{corollary}
\label{corls}
For every $q\geq 0$, there exists a finite constant $C(q)$ such that
\begin{eqnarray*}
\sup_{\pi\in E_q} \Big\{ |\!| \!| \partial_t (\rho,\bs p) |\!| \!|^2_{-1}
\;+\; \sum_{k=0}^d\int_0^T dt \int_{{D^d}} du\;
{\Vert\nabla p_k(t,u)\Vert^2}
\Big \} \;\leq\; C(q)\;. 
\end{eqnarray*}
\end{corollary}

Next result together with the previous estimates provide the
compactness needed in the proof of the lower semicontinuity of the
rate function.

\begin{lemma}
\label{lem02}
Let $\{\rho^n:n\geq 1\}$ be a sequence of functions in $L^2(\Omega_T)$
such that uniformly on $n$, 
\begin{equation*}
\int_0^T dt\left\Vert\rho^n_t\right\Vert^2_{1,2} + \int_0^T dt
\left\Vert\partial_t\rho^n_t\right\Vert_{-1}^2 < C 
\end{equation*}
for some positive constant $C$. Suppose that $\rho\in L^2(\Omega_T)$
and that $\rho^n \rightarrow \rho$ weakly in $L^2(\Omega_T)$. Then
$\rho_n\rightarrow \rho$ strongly in $L^2(\Omega_T)$.
\end{lemma}

\begin{proof}
Since $H^1({D^d})\subset L^2({D^d})\subset H^{-1}({D^d})$ with
compact embedding $H^1({D^d})\to L^2({D^d})$, from Corollary 8.4,
\cite{Si}, the sequence $\{\rho_n\}$ is relatively compact in
$L^2\big([0,T],L^2({D^d})\big)$. Therefore the weak convergence
implies the strong convergence in $L^2\big([0,T],L^2({D^d})\big)$.
\end{proof}

\begin{theorem}
\label{th4}
The functional $I_T(\cdot|\gamma)$ is lower semicontinuous and has
compact level sets. 
\end{theorem}

\begin{proof}
We have to show that, for all $q\geq 0$, $E_q$ is compact in
$D([0,T],\mc M)$. Since $E_q\subset C([0,T],\mc M^0)$ and $C([0,T],\mc
M^0)$ is a closed subset of $D([0,T],\mc M)$, we just need to show
that $E_q$ is compact in $C([0,T],\mc M^0)$.

We will show first that $E_q$ is closed in $C([0,T],\mc M^0)$. Fix
$q\in\bb{R}$ and let $\{\pi^n:\,n\geq 1\}$ be a sequence in $E_q$
converging to some $\pi$ in $C([0,T],\mc M^0)$. Then, for all $G\in\mc
C(\overline{\Omega_T})\times [\mc C(\overline{D^d})]^d$,
\begin{eqnarray*}
\lim_{n\to\infty}\int_0^T dt\;\langle\pi^n_t,G_t\rangle = \int_0^T
dt\;\langle\pi_t,G_t\rangle\, . 
\end{eqnarray*}
Notice that this means that $\pi^{n,k}\rightarrow\pi^k$ weakly in
$L^2(\Omega_T)$, for each $k=0,\ldots,d$, which together with Corollary \ref{corls} and Lemma
\ref{lem02} imply that $\pi^{n,k}\rightarrow\pi^k$ strongly in
$L^2(\Omega_T)$. From this fact and the definition of $J_G$ it is easy
to see that, for all $G$ in $\mc C_0^{1,2}(\overline{\Omega_T})\times [\mc C_0^2(\overline{D^d})]^d$,
\begin{equation*}
\lim_{n\to\infty}J_G(\pi_n) = J_G(\pi)\, .
\end{equation*}
This limit, Corollary \ref{corls} and the lower semicontinuity of $\mc
Q$ permit us to conclude that $\mc Q(\pi)\leq C(q)$ and that
$I_T(\pi|\gamma)\leq q$.

We prove now that $E_q$ is relatively compact. To this end, it is
enough to prove that for every continuous function 
$G:\overline{D^d}\to\bb R$, and every $k=0,\ldots,d$,
\begin{eqnarray}
\label{est4}
\lim_{\delta\to 0}\sup_{\pi\in E_q}\sup_{\substack{0\leq s,r\leq T\\
  |r-s|<\delta}}|\langle\pi_r^k,G\rangle-\langle\pi_s^k,G\rangle|=0\, . 
\end{eqnarray}
Since $E_q\subset C([0,T],\mc M^0)$, we may assume by approximations
of $G$ in $L^1({D^d})$ that $G\in\mc C_c^{\infty}({D^d})$. In which
case, \eqref{est4} follows from \eqref{cont}.
\end{proof}

We conclude this section with an explicit formula for the rate
function $I_T(\cdot |\gamma)$.  For each $\pi (t,du) = (\rho,\bs p)(t,u)du$ in
$D([0,T], \mc M^0)$, denote by $H^1_0(\pi)$ the Hilbert space
induced by $\mc C^{1,2}_0(\overline{\Omega_T})$ endowed with the inner
product $\langle\cdot,\cdot\rangle_{\pi}$ defined by
\begin{equation}\label{prodintpi}
\langle H,G\rangle_{\pi}
=\sum_{v\in\mc V}\int_0^Tdt\;\int_{{D^d}}du \chi(\theta_v(\Lambda(\rho,\bs p)))[\tilde{v}\cdot\nabla H][\tilde{v}\cdot\nabla G]
\,. 
\end{equation}
Induced means that we first declare two functions $F,G$ in $\mc
C^{1,2}_0(\overline{\Omega_T})$ to be equivalent if $\langle
F-G,F-G\rangle_{\pi} = 0$ and then we complete the quotient
space with respect to the inner product
$\langle\cdot,\cdot\rangle_{\pi}$. The norm of
$H^1_0(\pi)$ is denoted by $\Vert\cdot\Vert_{\pi}$.

Fix a path $\pi$ in $D([0,T], \mc M^0)$ and a function $H$ in
$H^1_0(\pi)$.  A measurable function $\lambda :
[0,T]\times{D^d} \to \bb R_{+}\times\bb R^d$ is said to be a weak solution of the
nonlinear boundary value parabolic equation
\begin{eqnarray}
\label{f05}
\begin{cases}
\partial_t \lambda\!\!\! &+ \sum_{i=1}^d\sum_{v\in \mathcal{ V}} \tilde{v} \partial_{x_i}\left[\chi(\theta_v(\Lambda(\lambda)))(v_i-\tilde{v}\cdot\partial_{x_i}H)\right] = \frac{1}{2}\Delta \lambda,\\
\lambda(0,\cdot)\!\!\! &=  \gamma(\cdot)\\
\lambda(t,x)\!\!\! &= d(x), x\in \{0,1\}\times\mathbb{T}^{d-1},
\end{cases}
\end{eqnarray}
if it satisfies the following two conditions.

\begin{enumerate}
\item[(i)] For $k=0,\ldots,d$, $\lambda_k$ belongs to $L^2 \left( [0,T] ,
    H^1({D^d})\right)$: 
\begin{equation*}
\int_0^T d s\Big( \int_{D^d} {\parallel\nabla \lambda_k(s,u)\parallel}^2 
du \Big)<\infty \; ;
\end{equation*}

\item[(ii)] For every function $G(t,u)=G_t(u)$ in ${\cc}^{1,2}_0
  (\overline{\Omega_T})$,
$$
\int_{{D^d}} G(T,u)\cdot \lambda(T,u)du - \int_{{D^d}} G(0,u)\cdot\gamma (u) du
$$
$$
=\int_0^T dt \int_{{D^d}} du \left\{ \lambda(t,u)\cdot\partial_t G(t,u) + \frac{1}{2} \lambda(t,u) \cdot \sum_{1\leq i\leq d} \partial_{u_i}^2 G(t,u)\right\}
$$
$$
-\int_0^T dt \int_{\{1\}\times\mathbb{ T}^{d-1}} dS\,\, b(\tilde{u})\cdot \partial_{u_1} G(t,u) + \int_0^T dt \int_{\{0\}\times\mathbb{ T}^{d-1}} dS\,\, a(\tilde{u})\cdot \partial_{u_1} G(t,u)
$$
$$
-\int_0^T dt \int_{{D^d}} du\,\,  \sum_{v\in\mathcal{ V}} \tilde{v}\cdot \chi(\theta_v(\Lambda(\lambda))) \sum_{1\leq i \leq d} v_i \partial_{u_i} G(t,u),
$$
$$+\sum_{v\in\mc V}\int_0^Tdt\;\int_{{D^d}}du \chi(\theta_v(\Lambda(\lambda)))[\tilde{v}\cdot\nabla H][\tilde{v}\cdot\nabla G].$$
\end{enumerate}

Uniqueness of solutions of equation \eqref{f05} follows from the same arguments of the uniqueness proved in \cite{s}.

\begin{lemma}
\label{lem05}
Assume that $\pi(t,du) = (\rho,\bs p)(t,u)du$ in $D([0,T],\mc M^0)$ has finite
rate function: $I_T(\pi|\gamma)<\infty$. Then, there exists a function
$H$ in $H^1_0(\pi)$ such that $(\rho,\bs p)$ is a weak solution to
\eqref{f05}.  Moreover,
\begin{eqnarray}
\label{f06}
I_T(\pi|\gamma) = \frac{1}{4}\Vert H\Vert_{\pi}^2\, .
\end{eqnarray}
\end{lemma}

The proof of this lemma is similar to the one of Lemma 10.5.3 in
\cite{KL} and is therefore omitted.

\section{$I_T(\cdot|\gamma)$-Density}
\label{sec5}

The main result of this section, stated in Theorem \ref{th5}, asserts
that any trajectory $\lambda_t$, $0\le t\le T$, with finite rate
function, $I_T(\lambda|\gamma)<\infty$, can be approximated by a
sequence of smooth trajectories $\{\lambda^n : n\ge 1\}$ such that
\begin{equation*}
\lambda^n \longrightarrow \lambda \quad\text{and}\quad 
I_T(\lambda^n|\gamma)  \longrightarrow  I_T(\lambda|\gamma)\;.
\end{equation*}
This is one of the main steps in the proof of the lower bound of the
large deviations principle for the empirical measure.  The proof is mainly based on the regularizing effects of the hydrodynamic
equation. This strategy was introduced by \cite{flm}.

A subset $A$ of $D([0,T],\mc M_{+}\times \mc M^d)$ is said to be
$I_T(\cdot|\gamma)$-dense if for every $\pi$ in $D([0,T],\mc M_{+}\times\mc M^d)$ such
that $I_T(\pi|\gamma)<\infty$, there exists a sequence $\{\pi^n : n\ge
1\}$ in $A$ such that $\pi^n$ converges to $\pi$ and
$I_T(\pi^n|\gamma)$ converges to $I_T(\pi|\gamma)$.

Let $\Pi_1$ be the subset of $D([0,T],\mc M^0)$ consisting of paths
$\pi(t,du) = (\rho,\bs p)(t,u)du$ whose density $(\rho,\bs p)$ is a weak solution of
the hydrodynamic equation \eqref{problema2} in the time interval
$[0,\delta]$ for some $\delta>0$.

\begin{lemma}\label{itdensidade1}
The set $\Pi_1$ is $I_T(\cdot|\gamma)$-dense.
\end{lemma}

\begin{proof}
Fix $\pi(t,du) = (\rho,\bs p)(t,u) du$ in $D([0,T],\mc M_{+}\times \mc M^d)$ such that
$I_T(\pi|\gamma)<\infty$. By Lemma \ref{lem01}, $\pi$ belongs to
$C([0,T],\mc M^0)$. For each $\delta>0$, let $(\rho^{\delta},\bs p^\delta)$ be the
path defined as
\begin{equation*}
(\rho^{\delta},\bs p^\delta)(t,u) = 
\begin{cases}
\tau(t,u) & \hbox{ if } 0\leq t\leq\delta\, , \\
\tau(2\delta-t,u) & \hbox{ if } \delta\leq t\leq 2\delta\, , \\
(\rho,\bs p)(t-2\delta,u) & \hbox{ if } 2\delta\leq t\leq T\, ,
\end{cases}
\end{equation*}
where $\tau$ is the weak solution of the hydrodynamic equation
\eqref{problema2} starting at $\gamma$.  It is clear that $\pi^\delta(t,du)
= (\rho^{\delta},\bs p^\delta)(t,u)du$ belongs to $D_{\gamma}$, because so do $\pi$
and $\tau$ and that $\mc Q(\pi^{\delta})\leq \mc Q(\pi) +2\mc
Q(\tau)<\infty$. Moreover, $\pi^{\delta}$ converges to $\pi$ as
$\delta\downarrow 0$ because $\pi$ belongs to $\mc C([0,T],\mc M^0)$. By
the lower semicontinuity of $I_T(\cdot|\gamma)$, $I_T(\pi|\gamma)\leq
\liminf_{\delta\to 0} I_T(\pi^{\delta}|\gamma)$. Then, in order to
prove the lemma, it is enough to prove that $I_T(\pi|\gamma)\geq
\limsup_{\delta\to 0} I_T(\pi^{\delta}|\gamma)$. To this end,
decompose the rate function $I_T(\pi^{\delta}|\gamma)$ as the sum of
the contributions on each time interval $[0,\delta]$,
$[\delta,2\delta]$ and $[2\delta,T]$. The first contribution vanishes
because $\pi^{\delta}$ solves the hydrodynamic equation in this
interval. On the time interval $[\delta,2\delta]$,
$\partial_t\rho^\delta_t = -\partial_t\tau_{2\delta-t}=
-\frac 12\Delta\tau_{2\delta-t} + \sum_{v\in\mc V}\tilde{v}[v\cdot\nabla\chi(\theta_v(\Lambda(\tau_{2\delta-t})))] =-\frac 12\Delta(\rho^\delta_t,\bs p^\delta_t)+ \sum_{v\in\mc V}\tilde{v}[v\cdot\nabla\chi(\theta_v(\Lambda(\rho^\delta_t,\bs p^\delta_t)))] $.
In particular, the second contribution is equal to
\begin{eqnarray*}
\sup_{G\in \mc C^{1,2}_0(\overline{\Omega_T})\times [\mc C(\overline{D^d})]^d}
\Big\{\sum_{i=1}^d\int_{0}^{\delta}ds\int_{{D^d}}du\;
\partial_{x_i}(\rho,\bs p)\cdot\partial_{x_i} G -
\sum_{v\in\mc V}\int_0^\delta dt\;\int_{{D^d}}du \chi(\theta_v(\Lambda(\rho,\bs p)))[\tilde{v}\cdot\nabla G]^2\Big\} 
\end{eqnarray*}
which, by Lemma \ref{energiachi} is bounded from above,  and therefore this last expression converges to zero as
$\delta\downarrow 0$. Finally, the third contribution is bounded by
$I_T(\pi|\gamma)$ because $\pi^{\delta}$ in this interval is just a
time translation of the path $\pi$.
\end{proof}

Let $\Pi_2$ be the set of all paths $\pi$ in $\Pi_1$ with the property
that for every $\delta>0$ there exists $\epsilon>0$ such that, for $k=0,\ldots,d$,
$d(\pi_t^k(\cdot),\partial\mf U)\geq\epsilon$ for all $t\in[\delta,T]$, where $\partial\mf U$ stands for the boundary of $\mf U$.

We begin by proving an auxiliary lemma.

\begin{lemma}\label{quasiconcavidade}
Let $\pi,\lambda\in \mf U$, and let $\pi^\epsilon = (1-\epsilon)\pi+\epsilon\lambda$, $0\le\epsilon\le 1$. Then, for all
$v\in\mc V$, we have
$$\theta_v(\Lambda(\pi^\epsilon)) = (1-\epsilon)\theta_v(\Lambda(\pi)) + \epsilon\theta_v(\Lambda(\lambda)).$$
\end{lemma}
\begin{proof}
Fix some $\lambda\in\mf U$. Observe that 
$$\left(\sum_{v\in\mc V}\theta_v(\Lambda(\lambda)),\sum_{v\in\mc V}v_1\theta_v(\Lambda(\lambda)),\ldots,\sum_{v\in\mc V}v_d\theta_v(\Lambda(\lambda)) \right) = (\lambda_0,\lambda_1,\ldots,\lambda_d)$$
is a linear system with $d+1$ equations and $|\mc V|$ unknowns (given by $\theta_v(\Lambda(\lambda))$, for $v\in\mc V$). Therefore, any solution of this linear system can be expressed as a linear combination of $\lambda_i$, $i=0,1,\ldots,d$. The proof follows from this fact.
\end{proof}
\begin{remark}
In the particular case when $d=1$ and the set of velocities is $\mc V = \{v,-v\}\subset\bb R$, a simple computation gives the unique solution
$$\theta_v(\Lambda(\lambda_0,\lambda_1)) = \frac{\lambda_0}{2} + \frac{\lambda_1}{2v}\quad\hbox{and}\quad \theta_{-v}(\Lambda(\lambda_0,\lambda_1))=\frac{\lambda_0}{2} - \frac{\lambda_1}{2v}.$$
\end{remark}

\begin{lemma}
The set $\Pi_2$ is $I_T(\cdot|\gamma)$-dense.
\end{lemma}

\begin{proof}
By Lemma \ref{itdensidade1}, it is enough to show that each path
$\pi(t,du)=(\rho,\bs p)(t,u) du$ in $\Pi_1$ can be approximated by paths in
$\Pi_2$.  Fix $\pi$ in $\Pi_1$ and let $\tau$ be as in the proof of
the previous lemma. For each $0<\varepsilon<1$, let
$(\rho^{\varepsilon},\bs p^{\varepsilon})=(1-\varepsilon)(\rho,\bs p)+\varepsilon\tau$,
$\pi^\varepsilon (t,du) = (\rho^{\varepsilon},\bs p^{\varepsilon})(t,u) du$.  Note that $\mc
Q(\pi^{\varepsilon})<\infty$ because $\mc Q$ is convex and both $\mc
Q(\pi)$ and $\mc Q(\tau)$ are finite.  Hence, $\pi^{\varepsilon}$
belongs to $D_{\gamma}$ since both $\rho$ and $\tau$ satisfy the
boundary conditions. Moreover, It is clear that $\pi^{\varepsilon}$
converges to $\pi$ as $\varepsilon\downarrow 0$.  By the lower
semicontinuity of $I_T(\cdot|\gamma)$, in order to conclude the proof,
it is enough to show that
\begin{eqnarray}\label{ls2}
\limsup_{N\to\infty}I_T(\pi^{\varepsilon}|\gamma)\leq I_T(\pi|\gamma)\, .
\end{eqnarray}

By Lemma \ref{lem05}, there exists $H\in H^1_0(\pi)$ such that
$(\rho, \bs p)$ solves the equation \eqref{f05}. Let us denote $\chi(\theta_v(\Lambda(\rho,\bs p)))$ simply by $\chi_v(\pi)$, and define 
$P_{i,v}(\pi) = \chi_v(\pi)\Big(\tilde{v}\cdot\partial_{x_i}H-v_i\Big)$, and note that $P_{i,v}(\tau) = -v_i\chi(\theta_v(\Lambda(\tau)))$. Let also

$$P_{i,v}^\epsilon = (1-\epsilon)P_{i,v}(\pi) + \epsilon P_{i,v}(\tau).$$

Observe that, by Lemma \ref{lem05},
$$I(\pi) = \frac{1}{4}\| H\|_\pi^2,$$
and that, using the definition of $\|\cdot\|_\pi$ in \eqref{prodintpi},
$$\frac{1}{4}\| H\|_\pi^2 = \frac{1}{4}\sum_{i,v} \int_0^Tdt\int_{D^d}du \chi_v(\pi)(\tilde{v}\cdot \partial_{x_i}H)^2= \frac{1}{4} \sum_{i,v}\int_0^Tdt\int_{D^d}du \frac{(P_{i,v}+v_i\chi_v(\pi))^2}{\chi_v(\pi)}.$$

A simple computation shows that
$$\bb J_G(\pi^\epsilon) = \sum_{i,v}\int_0^T\int_{D^d} [P_{i,v}^\epsilon + \chi_v(\pi^\epsilon)v_i](\tilde{v}\cdot \partial_{x_i}G)-\chi_v(\pi^\epsilon)(\tilde{v}\cdot \partial_{x_i}G)^2$$
$$=\frac{1}{4} \sum_{i,v} \int_0^Tdt\int_{D^d}du \frac{[P_{i,v}^\epsilon + \chi_v(\pi^\epsilon)v_i]^2}{\chi_v(\pi^\epsilon)} - \left(\frac{1}{2} \frac{P_{i,v}^\epsilon+\chi_v(\pi^\epsilon)}{\sqrt{\chi_v(\pi^\epsilon)}}-\sqrt{\chi_v(\pi)}(\tilde{v}\cdot \partial_{x_i}G)\right)^2.$$

Let
$$A_\epsilon = \frac{1}{4} \sum_{i,v} \int_0^Tdt\int_{D^d}du \frac{[P_{i,v}^\epsilon + \chi_v(\pi^\epsilon)v_i]^2}{\chi_v(\pi^\epsilon)},$$
and
$$B_\epsilon(G) = \int_0^Tdt\int_{D^d}du\left(\frac{1}{2} \frac{P_{i,v}^\epsilon+\chi_v(\pi^\epsilon)}{\sqrt{\chi_v(\pi^\epsilon)}}-\sqrt{\chi_v(\pi)}(\tilde{v}\cdot \partial_{x_i}G)\right).$$
This implies that 
$$I(\pi^\epsilon) = \sup_{G} \bb J_G(\pi^\epsilon) = \sup_{G} \left\{ A_\epsilon - B_\epsilon(G)^2\right\} = A_\epsilon - \inf_{G}  B_\epsilon(G)^2 \leq A_\epsilon,$$
where the supremum and infimum are taken over in $G$ in $C^{\infty}_c(\Omega_T)\times[\mc C^{\infty}_c(D^d)]^{d}$.

It remains to be shown that $A_\epsilon$ is uniformly integrable in $\epsilon$. However, this is a simple consequence of Lemma \ref{quasiconcavidade}.
\end{proof}

Let $\Pi$ be the subset of $\Pi_2$ consisting of all those paths $\pi$
which are solutions of the equation \eqref{f05} for some $H\in\mc
C^{1,2}_0(\overline{\Omega_T})\times [\mc C(\overline{D^d})]^d$.

\begin{theorem}
\label{th5}
The set $\Pi$ is $I_T(\cdot|\gamma)$-dense.
\end{theorem}

\begin{proof}
By the previous lemma, it is enough to show that each path $\pi$ in
$\Pi_2$ can be approximated by paths in $\Pi$. Fix $\pi(t,du)
=(\rho,\bs p)(t,u) du$ in $\Pi_2$.  By Lemma \ref{lem05}, there exists $H\in
H^1_{0}(\pi)$ such that $(\rho,\bs p)$ solves the equation
\eqref{f05}. Since $\pi$ belongs to $\Pi_2\subset \Pi_1$, $(\rho,\bs p)$ is the
weak solution of \eqref{problema2} in some time interval $[0,2\delta]$ for
some $\delta>0$. In particular, $\nabla H^k = 0$ a.e in
$[0,2\delta]\times{D^d}$. On the other hand, since $\pi$ belongs to
$\Pi_1$, there exists $\epsilon>0$ such that, for $k=0,\ldots,d$, $d(\pi_t^k (\cdot),\partial\mf U)\ge\epsilon$ for $\delta \le t\le T$.  Therefore,
\begin{equation}
\label{fin}
\int_0^T dt\int_{{D^d}} \Vert\nabla H_t(u)\Vert^2 \, du
\;<\; \infty\, .
\end{equation}

Since $H$ belongs to $H^1_{0}(\pi)$, there exists a sequence
of functions $\{H^n:\, n\geq 1\}$ in $\mc
C^{1,2}_{0}(\overline{\Omega_T})$ converging to $H$ in
$H^1_0(\pi)$. We may assume of course that $\nabla H^n_t
\equiv 0$ in the time interval $[0,\delta]$. In particular,
\begin{eqnarray}
\label{lim3}
\lim_{n\to\infty}\int_0^T dt\int_{{D^d}} du\; 
\Vert \nabla H^n_t(u)-\nabla H_t(u)\Vert^2 = 0\, .
\end{eqnarray}

For each integer $n>0$, let $(\rho^n,\bs p^n)$ be the weak solution of
\eqref{f05} with $H^n$ in place of $H$ and set $\pi^n (t,du) =
(\rho^n,\bs p^n)(t,u)du$. By \eqref{f06} and since $\chi$ is bounded above in
$[0,1]$ by $1/4$, we have that
\begin{equation*}
I_T(\pi^n|\gamma) = \frac{1}{2}\sum_{v\in\mc V}\int_0^T\ dt\;
\langle\chi(\theta_v(\Lambda(\rho^n_t,\bs p_t^n))),\Vert\nabla H^n_t\Vert^2\rangle
\leq C_0 \int_0^T dt\int_{{D^d}} du\;
\Vert\nabla H^n_t(u)\Vert^2\, .
\end{equation*}
In particular, by \eqref{fin} and \eqref{lim3}, $I_T(\pi^n|\gamma)$ is
uniformly bounded on $n$. Thus, by Theorem \ref{th4}, the sequence
$\pi^n$ is relatively compact in $D([0,T],\mc M_{+}\times\mc M^d)$.

Let $\{\pi^{n_k}:\, k\geq 1\}$ be a subsequence of $\pi^n$ converging
to some $\pi^0$ in $D([0,T],\mc M^0)$. For every $G$ in $\mc
C^{1,2}_0(\overline{\Omega_T})$,
$$
\int_{{D^d}} G(T,u)\cdot (\rho^{n_k}_t,\bs p_t^{n_k})(T,u)du - \int_{{D^d}} G(0,u)\cdot\gamma (u) du
$$
$$
=\int_0^T dt \int_{{D^d}} du \left\{ (\rho^{n_k}_t,\bs p_t^{n_k})(t,u)\cdot\partial_t G(t,u) + \frac{1}{2} (\rho^{n_k}_t,\bs p_t^{n_k})(t,u) \cdot \sum_{1\leq i\leq d} \partial_{u_i}^2 G(t,u)\right\}
$$
$$
-\int_0^T dt \int_{\{1\}\times\mathbb{ T}^{d-1}} dS\,\, b(\tilde{u})\cdot \partial_{u_1} G(t,u) + \int_0^T dt \int_{\{0\}\times\mathbb{ T}^{d-1}} dS\,\, a(\tilde{u})\cdot \partial_{u_1} G(t,u)
$$
$$
-\int_0^T dt \int_{{D^d}} du\,\,  \sum_{v\in\mathcal{ V}} \tilde{v}\cdot \chi(\theta_v(\Lambda(\rho^{n_k}_t,\bs p_t^{n_k}))) \sum_{1\leq i \leq d} v_i \partial_{u_i} G(t,u),
$$
$$+\sum_{v\in\mc V}\int_0^Tdt\;\int_{{D^d}}du \chi(\theta_v(\Lambda(\rho^{n_k}_t,\bs p_t^{n_k})))[\tilde{v}\cdot\nabla H^{n_k}][\tilde{v}\cdot\nabla G].$$

Letting $k\to\infty$ in this equation, we obtain the same equation
with $\pi^0$ and $H$ in place of $\pi^{n_k}$ and $H^{n_k}$,
respectively, if
\begin{equation}
\label{lim2}
\begin{split}
& \lim_{k\to\infty} \int_0^T dt \int_{{D^d}} du\,\,  \sum_{v\in\mathcal{ V}} \tilde{v}\cdot \chi(\theta_v(\Lambda(\rho^{n_k}_t,\bs p_t^{n_k}))) \sum_{1\leq i \leq d} v_i \partial_{u_i} G(t,u)\\
& \;=\; \int_0^T dt \int_{{D^d}} du\,\,  \sum_{v\in\mathcal{ V}} \tilde{v}\cdot \chi(\theta_v(\Lambda(\rho_t^0,\bs p_t^0))) \sum_{1\leq i \leq d} v_i \partial_{u_i} G(t,u), \\
&\lim_{k\to\infty} \sum_{v\in\mc V}\int_0^Tdt\;\int_{{D^d}}du \chi(\theta_v(\Lambda(\rho^{n_k}_t,\bs p_t^{n_k})))[\tilde{v}\cdot\nabla H^{n_k}][\tilde{v}\cdot\nabla G]\\
 &= \sum_{v\in\mc V}\int_0^Tdt\;\int_{{D^d}}du \chi(\theta_v(\Lambda(\rho^0_t,\bs p_t^0)))[\tilde{v}\cdot\nabla H][\tilde{v}\cdot\nabla G].\,
\end{split}
\end{equation}

We prove the second claim, the first one being simpler.  Note first
that we can replace $H^{n_k}$ by $H$ in the previous limit, because
$\chi$ is bounded in $[0,1]$ by $1/4$, and
\eqref{lim3} holds. Now, $(\rho^{n_k},\bs p^{n_k})$ converges to $(\rho^0,\bs p^0)$ weakly in
$L^2(\Omega_T)$ because $\pi^{n_k}$ converges to $\pi^0$ in
$D([0,T],\mc M^0)$. Since $I_T(\pi^n|\gamma)$ is uniformly bounded, by
Corollary \ref{corls} and Lemma \ref{lem02}, $(\rho^{n_k},\bs p^{n_k})$ converges to
$(\rho^0,\bs p^0)$ strongly in $L^2(\Omega_T)$ which implies \eqref{lim2}. In
particular, since \eqref{fin} holds, by uniqueness of weak solutions
of equation \eqref{f05}, $\pi^0 = \pi$ and we are done.
\end{proof}

\section{Large deviations}
\label{sec6}
We prove in this section Theorem \ref{mt}, which is the dynamical large deviations principle for
the empirical measure of boundary driven stochastic lattice gas model with many conserved quantities.
The proof uses some of the ideas introduced in \cite{flm}.

\subsection{Superexponential estimates}

It is well known that one of the main steps in the derivation of the
upper bound is a super-exponential estimate which allows the
replacement of local functions by functionals of the empirical density
in the large deviations regime. 

Let $\kappa$ be as in the beginning of Section \ref{sec4}. Note that since $\nu^N_{\kappa}$ is not the invariant state, there are no
reasons for $\langle-N^2\mc L_N f,f\rangle_{\nu^N_{\kappa}}$ to
be positive. The next statement shows that this expression is almost
positive.

For each function $f:X_N\to\bb R$, let $D_{\nu_{\kappa}^N}(f)$ be
$$D_{\nu_{\kappa}^N}(f) = D_{\nu_{\kappa}^N}^{ex}(f) + D_{\nu_{\kappa}^N}^c(f) + D_{\nu_{\kappa}^N}^b(f),$$
where
$$D_{\nu_{\kappa}^N}^{ex}(f) = \sum_{v\in\mathcal{V}}\sum_{x\in D_N^d} \sum_{x+z\in D_N^d}P_N(z-x,v)\int \left[\sqrt{f(\eta^{x,z,v})} - \sqrt{f(\eta)}\right]^2 \nu_\kappa^n(d\eta),$$
$$D_{\nu_{\kappa}^N}^{c}(f) = \sum_{q\in\mathcal{Q}}\sum_{x\in D_N^d} \int p(x,q,\eta)\left[\sqrt{f(\eta^{x,q})} - \sqrt{f(\eta)} \right]^2\nu_\kappa^N(d\eta),$$
and
\begin{align*}
D_{\nu_{\kappa}^N}^{b}(f) = \sum_{v\in\mathcal{V}}\sum_{x\in\{1\}\times\mathbb{T}_N^{d-1}} \int &[\alpha_v(\tilde{x}/N)(1-\eta(x,v)) + (1-\alpha_v(\tilde{x}/N))\eta(x,v)]\times\\
&\times\left[\sqrt{f(\sigma^{x,v}\eta)}-\sqrt{f(\eta)}\right]^2\nu_\kappa^N(d\eta) \; +\\
+\;\sum_{v\in\mathcal{V}}\sum_{x\in\{N-1\}\times\mathbb{T}_N^{d-1}}\int &[\beta_v(\tilde{x}/N)(1-\eta(x,v)) + (1-\beta_v(\tilde{x}/N))\eta(x,v)]\times\\
&\times\left[\sqrt{f(\sigma^{x,v}\eta)}-\sqrt{f(\eta)}\right]^2\nu_\kappa^N(d\eta).
\end{align*}

\begin{proposition}\label{dirichlet}
There exist constants $C_1>0$ and $C_2 = C_2(\alpha,\beta)>0$ such that for every density $f$ with respect to $\nu_\kappa^N$, then
$$<{\mathcal{ L}}_N\sqrt{f},\sqrt{f}>_{\nu_\kappa^N} \leq -C_1 D_{\nu_\kappa^N}(f) + C_2 N^{d-2}.$$
\end{proposition}
The proof of this proposition is elementary and is thus omitted. 

Further, we
may choose $\kappa$ for which there exists a constant $\theta>0$ such
that:
\begin{eqnarray*}
\kappa(u_1,\tilde{u}) =  d(-1,\tilde{u}) & 
\qquad\hbox{ if } \; 0\leq u_1\leq \theta\, , \\ 
\kappa(u_1,\tilde{u}) =  d(1,\tilde{u})\;\;\, & 
\,\hbox{ if } \;1-\theta\leq u_1\leq 1\, ,
\end{eqnarray*}
for all $\tilde{u}\in\bb T^{d-1}$. In that case, for every $N$ large
enough, $\nu_{\kappa}^N$ is reversible for the process with
generator $\mc L_{N}^b$ and then $\langle -N^2\mc L_{N}^b f,
f\rangle_{\nu^N_{\kappa}}$ is positive.

Fix $L\geq 1$ and a configuration $\eta$, let ${\boldsymbol I}^L(x,\eta):={\boldsymbol I}^L(x) = (I_0^L(x),\ldots,I_d^L(x))$ be the average of the conserved quantities in a cube of the length $L$ centered at $x$:
$${\boldsymbol I}^L(x)= \frac{1}{|\Lambda_L|} \sum_{z\in x+\Lambda_L} {\boldsymbol I}(\eta_z),$$
where, $\Lambda_L = \{-L,\ldots,L\}^d$ and $|\Lambda_L| = (2L+1)^d$ is the discrete volume of box $\Lambda_L$.

For each $G\in\mc C(\overline{\Omega_T})\times C(\overline{D^d})^d$, and each $\varepsilon>0$, let
\begin{equation*}
V_{N\varepsilon}^{G,1}(s,\eta)=\frac{1}{N^d}\sum_{k=0}^d\sum_{i,j=1}^d\sum_{x\in D_N^d}
\partial_{u_i}G^k(s,x/N)\left[\tau_x \tilde{V}_{N\varepsilon}^{j,k}\right]\, ,
\end{equation*}
where
\begin{align*}
\tilde{V}_{N\varepsilon}^{j,k}(\eta) &= \frac{1}{(2\ell +1)^d}\sum_{y\in{\Lambda_{N\varepsilon}}} \sum_{v\in\mathcal{ V}} v_k \sum_{z\in\mathbb{ Z}^d} p(z,v)z_j\tau_y( \eta(0,v)[1-\eta(z,v)] ) \\
&- \sum_{v\in\mathcal{ V}} v_j v_k \chi(\theta_v(\Lambda({\boldsymbol I}^\ell(0)))),
\end{align*}

and let

\begin{eqnarray*}
V_{N\varepsilon}^{G,2}(s,\eta)=\frac 1{2N^d}\sum_{v\in \mc V}\sum_{x\in D_N^d}\sum_{i=1}^d\sum_{j,k=0}^d v_k v_j\partial^N_{u_i}G^j_t(x/N) \partial^N_{u_i}G^k_t(x/N )\times\\
\times \left\{\eta(x,v)[1-\eta(x+e_i,v)] + \eta(x,v)[1- \eta(x-e_i,v)] - 2\chi(\theta_v(\Lambda({\boldsymbol I}^\ell(0))))\right\}
\end{eqnarray*}
Let, again, $G:[0,T]\times\mathbb{T}^{d-1}\to\mathbb{R}^{d+1}$ be a continuous function, and consider the quantities
\begin{equation*}
V_N^{-}(s,\eta,G) = \frac{1}{N^{d-1}}\sum_{k=0}^d\sum_{\tilde{x}\in\mathbb{T}_N^{d-1}} G_k(s,\tilde{x}/N) \Big(I_k(\eta_{(1,\tilde{x})}(s))- \sum_{v\in\mathcal{V}}v_k\alpha_v(\tilde{x}/N)\Big),
\end{equation*}
\begin{equation*}
V_N^{+}(s,\eta,G) = \frac{1}{N^{d-1}}\sum_{k=0}^d\sum_{\tilde{x}\in\mathbb{T}_N^{d-1}} G_k(s,\tilde{x}/N) \Big(I_k(\eta_{(N-1,\tilde{x})}(s))- \sum_{v\in\mathcal{V}}v_k\beta_v(\tilde{x}/N)\Big),
\end{equation*}
\begin{proposition}
\label{see}
Fix $G\in\mc C(\overline{\Omega_T})\times[\mc C(\overline{D^d})]^d$, $H$ in $\mc
C([0,T]\times\Gamma)\times[\mc C(\Gamma)]^d$, a cylinder function $\Psi$ and a sequence
$\{\eta^N: N\geq 1\}$ of configurations with $\eta^N$ in $X_N$. For
every $\delta>0$,
\begin{eqnarray*}
&&\limsup_{\varepsilon\to 0}\limsup_{N\to\infty}
\frac{1}{N^d}\, \log \bb{P}_{\eta^N}
\Big[ \, \Big|\int_0^T V_{N\varepsilon}^{G,j}(s,\eta_s) \, ds \Big|
>\delta\Big] \;=\; -\infty\, , \\
&&\limsup_{N\to\infty}\frac{1}{N^d} \, \bb P_{\eta^N}
\Big[ \, \Big|\int_0^T V_{N}^{\pm}(s,\eta,G)\Big|>\delta \, \Big] \;=\; -\infty\, ,
\end{eqnarray*}
for $j=1,2$.
\end{proposition}

The proof of the above proposition follows from Proposition \ref{dirichlet}, the replacement lemmas proved in \cite{s}, and the computation presented in \cite{BKL}, p.
78, for nonreversible processes. 

For each $\varepsilon> 0$ and $\pi$ in $\mc M_{+}\times\mc M^d$, for $k=0,\ldots,d$, denote by
$\Xi_\varepsilon (\pi_k) = \pi_k^{\varepsilon}$ the absolutely continuous
measure obtained by smoothing the measure $\pi_k$:
\begin{equation*}
\Xi_\varepsilon (\pi_k) (dx) \;=\; \pi_k^{\varepsilon} (dx) \;=\; 
\frac 1{U_\varepsilon} \frac {\pi_k(\bs \Lambda_\varepsilon(x))}
{|\bs \Lambda_\varepsilon(x)|} \,\, dx\;,
\end{equation*}
where $\bs \Lambda_\varepsilon(x) = \{y\in{D^d} : |y-x|\le
\varepsilon\}$, $|A|$ stands for the Lebesgue measure of the set $A$,
and $\{U_\varepsilon : \varepsilon >0\}$ is a strictly decreasing
sequence converging to $1$: $U_\varepsilon >1$, $U_\varepsilon >
U_{\varepsilon'}$ for $\varepsilon>\varepsilon'$, $\lim_{\varepsilon
  \downarrow 0} U_\varepsilon = 1$. Let 
\begin{equation*}
\pi^{N,\varepsilon} \;=\; \Big( \Xi_\varepsilon (\pi_0^N), \Xi_\varepsilon (\pi_1^N),\ldots,\Xi_\varepsilon (\pi_d^N)\Big).
\end{equation*}
A simple computation shows that $\pi^{N,\varepsilon}$ belongs to $\mc
M^0$ for $N$ sufficiently large because $U_\varepsilon >1$, and that
for each continuous function $H:{D^d} \to \bb R^{d+1}$,
\begin{equation*}
\<\pi^{N,\varepsilon}, H\> \;=\; \frac 1{N^d} \sum_{x\in
  {D_N^d}} H(x/N)\cdot {\bs I}^{\varepsilon N}(x) \; +\; O(N, \varepsilon)\;,
\end{equation*}
where $O(N, \varepsilon)$ is absolutely bounded by $C_0 \{ N^{-1} +
\varepsilon\}$ for some finite constant $C_0$ depending only on $H$.

For each $H$ in $\mc C_0^{1,2}(\overline{\Omega_T})\times \mc [C_0^{2}(\overline{D^d})]^d$ consider the
exponential martingale $M_t^H$ defined by
\begin{eqnarray*}
M_t^H &=& \exp\Big\{N^d \Big[\big\langle\pi_t^N,H_t\big\rangle
-\big\langle\pi_0^N,H_0\big\rangle \\ 
&& \qquad\qquad\; -\; \frac{1}{N^d} \int_0^t
e^{-N^d\langle\pi_s^N,H_s\rangle} \, \big(\partial_s + 
N^2\mc L_N\big) \, e^{N^d\langle\pi_s^N,H_s\rangle} \, ds \Big]\Big\}\, .
\end{eqnarray*}
Recall from subsection 2.2 the definition of the functional $\hat
J_H$. An elementary computation shows that
\begin{eqnarray}
\label{mart}
M_T^H = \exp\left\{N^d\left[\hat J_{H}(\pi^{N, \varepsilon})
+\bb V_{N,\varepsilon}^H + c^1_H(\varepsilon) 
+ c^2_H(N^{-1})\right]\right\}\, .
\end{eqnarray}
In this formula, 
\begin{equation*}
\begin{split}
\bb V_{N,\varepsilon}^H \; &=\; -\int_0^T 
V_{N\varepsilon}^{G,1}(s,\eta) \, ds
- \sum_{i=1}^d \int_0^T 
V_{N\varepsilon}^{G,2}(s,\eta) \, ds \\
\; & +\; \,V^+_N(s,\eta,\partial_{u_1}H)
\; -\; \, V^-_N(s,\eta,\partial_{u_1}H) 
\;+\; \langle\pi^N_0,H_0\rangle - \langle\gamma,H_0\rangle\, ;
\end{split}
\end{equation*}
and $c^j_H:\bb R_+ \to\bb R$, $j=1,2$, are functions depending only on
$H$ such that $c^j_H(\delta)$ converges to $0$ as $\delta\downarrow
0$.  In particular, the martingale $M_T^H$ is bounded by
$\exp\big\{C(H,T)N^d\big\}$ for some finite constant $C(H,T)$
depending only on $H$ and $T$. Therefore, Proposition \ref{see} holds
for $\bb P_{\eta^N}^H = \bb P_{\eta^N}M_T^H$ in place of $\bb
P_{\eta^N}$.

\subsection{Energy estimates}

To exclude paths with infinite energy in the large deviations regime,
we need an energy estimate. We state first the following technical
result.
\begin{lemma}
\label{est6}
There exists a finite constant $C_0$, depending on $T$, such that for
every $G$ in $C^{\infty}_c(\Omega_T)$, every integer $1\leq i\leq d$, $0\leq k\leq d$,
and every sequence $\{\eta^N: N\geq 1\}$ of configurations with
$\eta^N$ in $X_N$,
\begin{eqnarray*}
\limsup_{N\to\infty}\frac{1}{N^d}\log\bb E_{\eta^N}
\Big[\exp\Big\{N^d\int_0^Tdt\; \langle\pi^{N,k}_t,
\partial_{u_i}G\rangle\Big\}\Big] \; \leq \; 
C_0\Big\{1+\int_0^T \Vert G_t\Vert^2_2 \, dt \Big\}\; .
\end{eqnarray*}
\end{lemma}

The proof of this proposition follows from Lemma 3.8 in \cite{s}, and the fact that $d\delta_{\eta^N}/d\nu_{\kappa}^N \leq C^{N^d}$, for some positive constant $C=C(\kappa)$.

For each $G$ in $\mc
C^{\infty}_c(\Omega_T)$ and each integer $1\leq i\leq d$, let
$\tilde{\mc Q}_{i,k}^G: D([0,T],\mc M_{+}\times \mc M^d)\to \bb R$ be the function given by
\begin{equation*}
\tilde{\mc Q}_{i,k}^G(\pi) =
\int_0^Tdt\;\langle\pi_t^k,\partial_{u_i}G_t\rangle
-C_0\int_0^Tdt\int_{{D^d}} du\;G(t,u)^2 \;.
\end{equation*}

Notice that
\begin{eqnarray}
\label{f07}
\sup_{G\in\mc C^{\infty}_c(\Omega_T)}\left\{\tilde{\mc
    Q}_{i,k}^G(\pi)\right\} = \frac{\mc Q_{i,k}(\pi)}{4C_0}\, . 
\end{eqnarray}

Fix a sequence $\{G_r: r\geq 1\}$ of smooth functions dense in
$L^2([0,T], H^1({D^d}))$. For any positive integers $m,l$, let
\begin{equation*}
B_{m,l}^k = \Big\{\pi\in D([0,T],\mc M_{+}\times \mc M^d): \,\max_{\substack{1\leq j\leq
      m\\ 1\leq i\leq d}} \tilde{\mc Q}^{G_j}_{i,k}(\pi) \leq l\Big\}\, . 
\end{equation*}
Since, for fixed $G$ in $\mc C^{\infty}_c(\Omega_T)$ and $1\leq i\leq
d$ integer, the function $\tilde{\mc Q}_{i,k}^G$ is continuous, $B_{m,l}$
is a closed subset of $D([0,T],\mc M)$.

\begin{lemma}
\label{lemmaenergiald}
There exists a finite constant $C_0$, depending on $T$,  such that
for any positive integers $r,l$ and any sequence $\{\eta^N: N\geq 1\}$
of configurations with $\eta^N$ in $X_N$,
\begin{equation*}
\limsup_{N\to\infty} \frac{1}{N^d} 
\log Q_{\eta^N}\left[(B_{m,l}^k)^c\right] \leq -l+C_0 ,
\end{equation*}
where $k=0,\ldots,d$.
\end{lemma}

\begin{proof}
For integers $1\leq k\leq r$ and $1\leq i\leq d$, by Chebychev
inequality and by Lemma \ref{est6},
\begin{equation*}
\limsup_{N\to\infty} \frac{1}{N^d} 
\log \bb P_{\eta^N}\left[\tilde{\mc Q}_{i,k}^{G_m} > l\right] \leq -l +C_0\,.
\end{equation*}
Hence, from
\begin{eqnarray}
\label{ls}
\limsup_{N\to\infty} \frac{1}{N^d}\log(a_N+b_N) 
\leq \max\left\{\limsup_{N\to\infty} 
\frac{1}{N^d}\log a_N,\limsup_{N\to\infty} \frac{1}{N^d}\log b_N\right\}\, ,
\end{eqnarray}
we obtain the desired inequality.
\end{proof}

\begin{lemma}
\label{energiachi}
There exists a finite constant $C_0$, depending on $T$, such that for
every $G$ in $C^{\infty}_c(\Omega_T)\times [C^{\infty}_c(D^d)]^d$, and every sequence $\{\eta^N: N\geq 1\}$ of configurations with
$\eta^N$ in $X_N$,
\begin{eqnarray*}
\limsup_{N\to\infty}\frac{1}{N^d}\log\bb E_{\nu_{\kappa}^N}
\Big[\exp\Big\{N^d\int_0^T\sum_{i=1}^d\sum_{k=0}^d dt\; \langle\pi^{N}_t,
\partial_{u_i}G^k\rangle\Big\}\Big] \; \leq \; 
C_0\Big\{1+\int_0^T \Vert G_t\Vert^2_\pi \, dt \Big\}\; .
\end{eqnarray*}
In particular, we have that if $(\rho,\bs p)$ is the solution of \eqref{problema2}, then
\begin{eqnarray*}
\sup_{G\in \mc C^{1,2}_0(\overline{\Omega_T})}
\Big\{\sum_{i=1}^d\int_{0}^{T}ds\int_{{D^d}}du\;
\partial_{x_i}(\rho,\bs p)\cdot\partial_{x_i} G -
\sum_{v\in\mc V}\int_0^T dt\;\int_{{D^d}}du \chi(\theta_v(\Lambda(\rho,\bs p)))[\tilde{v}\cdot\nabla G]^2\Big\},
\end{eqnarray*}
is finite, and vanishes if $T\to 0$.
\end{lemma}
\begin{proof}
Applying Feynman-Kac's formula and using the same arguments of Lemma 3.3 in \cite{s}, we have that 
$$\frac{1}{N^d} \log E_{\nu_\kappa^N} \left[ \exp\left\{N \int_0^T ds \sum_{i=1}^d\sum_{k=0}^d\sum_{x\in D_N^{d}} (I_k(\eta_x(s))-I_k(\eta_{x-e_i}(s)))\partial_{u_i}G^k(s,x/N) \right\} \right]$$
is bounded above by
$$\frac{1}{N^d} \int_0^T \lambda_s^N ds,$$
where $\lambda_s^N$ is equal to
$$\sup_f \Big\{\Big< N\!\!\sum_{i,k}\sum_{x\in D_N^d}\!\! (I_k(\eta(x))-I_k(\eta(x-e_i)))\partial_{u_i}G^k(s,x/N),f\Big>_{\!\!\!\nu_\kappa^N}\!\!\!\! + N^2<{\mathcal{L}}_N\sqrt{f},\sqrt{f}>_{\nu_\kappa^N}\Big\},$$
where the supremum is taken over all densities $f$ with respect to $\nu_\kappa^N$. By Proposition \ref{dirichlet}, the expression inside brackets is bounded above by
$$CN^d - \frac{N^2}{2} D_{\nu_\kappa^N}(f) + \sum_{i,k}\sum_{x\in D_N^d}\left\{ N\partial_{u_i}G^k(s,x/N)\int[I_k(\eta_x)-I_k(\eta_{x-e_i})]f(\eta)\nu_\kappa^N(d\eta) \right\}.$$
We now rewrite the term inside the brackets as
$$ \sum_{v\in\mathcal{V}} \sum_{i=1}^d\sum_{x\in D_N^d}\left\{\int N (\tilde{v}\cdot\partial_{u_i}G(s,x/N))  [\eta(x,v) - \eta(x-e_i,v)]f(\eta) \nu_\kappa^N(d\eta)\right\}.$$
Writing $\eta(x,v) - \eta(x-e_i,v) = \eta(x,v)[1-\eta(x-e_i,v)] - \eta(x-e_i,v)[1-\eta(x,v)]$, and applying the same arguments in Lemma 3.8 of \cite{s}, we obtain that
$$N(\tilde{v}\cdot \partial_{u_i}G(s,x/N)) \int [\eta(x,v) - \eta(x-e_i,v)] f(\eta) \nu_\kappa^N(d\eta)$$
\begin{eqnarray*}
&\leq&(\tilde{v}\cdot \partial_{u_i}G(s,x/N))^2 \int \eta(x,v)[1-\eta(x-e_i,v)] f(\eta^{x-e_i,x,v})d\nu_{\kappa}^N\\
&+& \frac{1}{4} \int f(\eta^{x-e_i,x,v})\left[N\left(1-\frac{\gamma_{x-e_i},v}{\gamma_{x,v}}\right)\right]^2 \nu_\kappa^N(d\eta)\\
&+&N^2 \int \frac{1}{2} [\sqrt{f(\eta^{x-e_i,x,v})}- \sqrt{f(\eta)}]^2\nu_\kappa^N(d\eta)\\
&+& 2(\tilde{v}\cdot \partial_{u_i}G(s,x/N))^2\int \eta(x,v)[1-\eta(x-e_i,v)] (\sqrt{f(\eta)}+\sqrt{f(\eta^{x-e_i,x,v})})^2\nu_\kappa^N(d\eta),
\end{eqnarray*}
we have that $(\sqrt{f(\eta)}+\sqrt{f(\eta^{x-e_i,x,v})})^2\leq 2 (f(\eta)+f(\eta^{x-e_i,x,v}))$. An application of the replacement lemma (Lemma 3.7 in \cite{s}) concludes the proof.
\end{proof}

\subsection{Upper Bound}

Fix a sequence $\{F_j: j\geq 1\}$ of smooth functions
dense in $\mc C(\overline{{D^d}})$ for the uniform topology, with positive coordinates.  For
$j\ge 1$ and $\delta>0$, let
\begin{equation*}
D_{j,\delta} = \Big\{\pi\in D([0,T],\mc M_{+}\times\mc M^d): \, |\<\pi_t^k , F_j\>|
\le \breve{v}^k|\mc V|\int_{{D^d}} F_j(x) \, dx  \,+\, C_j\delta\;,\,k=0,\ldots,d,\, 0\le t\le T \Big\} \, ,
\end{equation*}
where $\breve{v}^0 = 1$ and $\breve{v}^k=\breve{v}$, $C_j = \Vert \nabla F_j \Vert_\infty$ and $\nabla F$ is the
gradient of $F$. Clearly, the set $D_{j,\delta}$, $j\ge 1$, $\delta
>0$, is a closed subset of $D([0,T],\mc M_{+}\times\mc M^d)$. Moreover, if
\begin{equation*}
E_{m,\delta} \;=\; \bigcap_{j=1}^m D_{j,\delta}\;,
\end{equation*}
we have that $D([0,T],\mc M^0) = \cap_{n\ge 1} \cap_{m\ge 1}
E_{m,1/n}$. Note, finally, that for all $m\ge 1$, $\delta>0$,
\begin{equation}
\label{ff01}
\pi^{N,\varepsilon} \text{ belongs to $E_{m,\delta}$ for $N$ sufficiently
large.}
\end{equation}
\smallskip

Fix a sequence of configurations $\{\eta^N: N\geq 1\}$ with $\eta^N$
in $X_N$ and such that $\pi^N(\eta^N)$ converges to $\gamma(u)du$ in
$\mc M$. Let $A$ be a subset of $D([0,T],\mc M_{+}\times\mc M^d)$,
\begin{equation*}
\frac{1}{N^d}\log\bb P_{\eta^N}\left[\pi^N\in A\right] 
= \frac{1}{N^d}\log \bb E_{\eta^N}\left[M_T^H \, (M_T^H)^{-1}
\, {\bf 1} \{\pi^N\in A\}\right]\, .
\end{equation*}
Maximizing over $\pi^N$ in $A$, we get from \eqref{mart} that the last
term is bounded above by
\begin{equation*}
-\inf_{\pi\in A} \hat J_H(\pi^{\varepsilon})
+\frac{1}{N^d}\log\bb E_{\eta^N}
\Big[M_T^H \, e^{-N^d\bb V_{N,\varepsilon}^H}\Big] 
- c^1_H(\varepsilon)-c^2_H(N^{-1})\, .
\end{equation*}
Since $\pi^N(\eta^N)$ converges to $\gamma(u)du$ in $\mc M$ and since
Proposition \ref{see} holds for $\bb P_{\eta^N}^H = \bb
P_{\eta^N}M_T^H$ in place of $\bb P_{\eta^N}$, the second term of the
previous expression is bounded above by some $C_H(\varepsilon, N)$
such that
\begin{equation*}
\limsup_{\varepsilon\to 0}\limsup_{N\to \infty}
C_H(\varepsilon, N) = 0\, .
\end{equation*}
Hence, for every $\varepsilon>0$, and every $H$ in $\mc
C^{1,2}_0(\overline{\Omega_T})\times \mc [C_0^{2}(\overline{D^d})]^d$,
\begin{eqnarray}
\label{est7}
\limsup_{N\to\infty}\frac{1}{N^d}\log\bb P_{\eta^N}[A]
\leq -\inf_{\pi\in A} \hat J_H(\pi^{\varepsilon})+C'_H(\varepsilon)\, ,
\end{eqnarray}
where $\displaystyle \lim_{\varepsilon\to 0}C'_H(\varepsilon)=0$.
Let
$$B_{r,l} = \Big\{\pi\in D([0,T],\mc M_{+}\times \mc M^d): \,\max_{\substack{1\leq j\leq
      r\\ 1\leq i\leq d}} \sum_{k=0}^d\tilde{\mc Q}^{G_j}_{i,k}(\pi) \leq l\Big\}, $$
and, for each $H\in \mc C_0^{1,2}(\overline{\Omega_T})\times \mc [C_0^{2}(\overline{D^d})]^d$, each
$\varepsilon>0$ and any $r,l, m, n \in\bb Z_+$, let
$J_{H,\varepsilon}^{r,l,m,n}:D([0,T],\mc M_{+}\times\mc M^d)\to\bb R\cup\{\infty\}$ be
the functional given by
\begin{equation*}
J_{H,\varepsilon}^{r,l,m,n}(\pi) = 
\begin{cases}
\hat J_H(\pi^{\varepsilon}) & \hbox{ if } \pi\in B_{r,l} \cap
E_{m,1/n} \, ,\\
+\infty & \hbox{ otherwise } .
\end{cases}
\end{equation*}
This functional is lower semicontinuous because so is $\hat J_H \circ
\Xi_\varepsilon$ and because $B_{r,l}$, $E_{m,1/n}$ are closed subsets
of $D([0,T],\mc M_{+}\times\mc M^d)$.

Let $\mc O$ be an open subset of $D([0,T],\mc M_{+}\times\mc M^d)$. By Lemma
\ref{lemmaenergiald}, \eqref{ls}, \eqref{ff01} and \eqref{est7},
\begin{eqnarray*}
\limsup_{N\to\infty}\frac{1}{N^d}\log Q_{\eta^N}[\mc O] 
& \leq & \max\Big\{\limsup_{N\to\infty}\frac{1}{N^d}
\log Q_{\eta^N}[\mc O\cap B_{r,l} \cap E_{m,1/n}] \, , \\ 
&& \qquad\qquad\qquad\qquad
\;\limsup_{N\to\infty}\frac{1}{N^d}\log Q_{\eta^N}[(B_{r,l})^c]\Big\}
\\ & \leq & \max\Big\{-\inf_{\pi\in \mc O\cap B_{r,l} \cap E_{m,1/n}}
\hat J_H(\pi^{\varepsilon}) + C'_H(\varepsilon)\, , \,-l+C_0\Big\}
\\ & = & - \inf_{\pi\in\mc O} L_{H,\varepsilon}^{r,l,m,n}(\pi) \, ,
\end{eqnarray*}
where
\begin{equation*}
L_{H,\varepsilon}^{r,l,m,n}(\pi)=\min \left\{J_{H,\varepsilon}^{r,l,m,n}(\pi)
- C'_H(\varepsilon)\, , \,l- C_0\right\}\, .
\end{equation*}
In particular,
\begin{equation*}
\limsup_{N\to\infty}\frac{1}{N^d}\log Q_{\eta^N}[\mc O]
\leq - \sup_{H,\varepsilon,r,l,m,n}\;\inf_{\pi\in\mc O}
L_{H,\varepsilon}^{r,l,m,n}(\pi)\, .
\end{equation*}

Note that, for each $H\in \mc C_0^{1,2}(\overline{\Omega_T})\times \mc [C_0^{2}(\overline{D^d})]^d$, each
$\varepsilon>0$ and $r,l,m,n \in\bb Z_+$, the functional
$L_{H,\varepsilon}^{r,l,m,n}$ is lower semicontinuous. Then, by Lemma
A2.3.3 in \cite{KL}, for each compact subset $\mc K$ of $D([0,T],\mc
M)$,
\begin{eqnarray*}
\limsup_{N\to\infty}\frac{1}{N^d}\log Q_{\eta^N}[\mc K] 
\;\leq\; - \inf_{\pi\in\mc K}\;\sup_{H,\varepsilon,r,l,m,n}
L_{H,\varepsilon}^{r,l,m,n}(\pi) \, .
\end{eqnarray*}
By \eqref{f07} and since $D([0,T],\mc M^0) = \cap_{n\ge 1}
\cap_{m\ge 1} E_{m,1/n}$,
\begin{equation*}
\begin{split}
& \limsup_{\varepsilon\to 0}\limsup_{l\to\infty}
\limsup_{r\to\infty}\limsup_{m\to\infty}\limsup_{n\to\infty} 
L_{H,\varepsilon}^{r,l,m,n}(\pi) \; = \\
&\qquad\qquad\qquad\qquad\qquad 
\begin{cases}
\hat J_H(\pi) & \hbox{ if } \mc Q(\pi)<\infty \text{ and } \pi \in
D([0,T],\mc M^0)\, ,\\
+\infty & \hbox{ otherwise }.
\end{cases}
\end{split}
\end{equation*}
This result and the last inequality imply the upper bound for compact
sets because $\hat J_H$ and $J_H$ coincide on $D([0,T],\mc M^0)$.  To
pass from compact sets to closed sets, we have to obtain exponential
tightness for the sequence $\{Q_{\eta^N}\}$. This means that there
exists a sequence of compact sets $\{\mc K_n : \,n\geq 1\}$ in
$D([0,T],\cm)$ such that
\begin{equation*}
\limsup_{N\to\infty}\frac{1}{N^d}\log Q_{\eta^N}({\mc K_n}^c)\leq -n\, .
\end{equation*}
The proof presented in \cite{B} for the non interacting zero range
process is easily adapted to our context.

\subsection{Lower Bound}

The proof of the lower bound is similar to the one in the convex
periodic case. We just sketch it and refer to \cite{KL}, Section 10.5.
Fix a path $\pi$ in $\Pi$ and let $H\in\mc
C^{1,2}_0(\overline{\Omega_T})$ be such that $\pi$ is the weak
solution of equation \eqref{f05}. Recall from the previous section the
definition of the martingale $M_t^H$ and denote by $\bb P^H_{\eta^N}$
the probability measure on $D([0,T],X_N)$ given by $\bb
P_{\eta^N}^H[A] = \bb E_{\eta^N}[M_T^H \mb 1 \{A\}]$. Under $\bb
P^H_{\eta^N}$ and for each $0\leq t\leq T$, the empirical measure
$\pi^N_t$ converges in probability to $\pi_t$. Further,
\begin{equation*}
\lim_{N\to\infty}\frac{1}{N^d}H
\left(\bb P^H_{\eta^N}\big|\bb P_{\eta^N}\right) = I_T(\pi|\gamma)\, ,
\end{equation*}
where $H(\mu|\nu)$ stands for the relative entropy of $\mu$ with
respect to $\nu$. From these two results we can obtain that for every
open set $\mc O\subset D([0,T],\mc M_{+}\times\mc M^d)$ which contains $\pi$,
\begin{equation*}
\liminf_{N\to\infty}\frac{1}{N^d}\log
\bb P_{\eta^N}\big[\mc O\big]\geq -I_T(\pi|\gamma)\, .
\end{equation*}
The lower bound follows from this and the $I_T(\cdot|\gamma)$-density
of $\Pi$ established in Theorem \ref{th5}.
\section*{Acknowledgements} We would like to thank Claudio Landim for suggesting this problem. 

\end{document}